\documentclass[11pt]{amsart}
\usepackage{amsthm}
\usepackage{amssymb}
\usepackage{mathrsfs}
\usepackage[margin=1in]{geometry}
\usepackage{color}
\usepackage{enumitem}
\usepackage{bbm}
\usepackage{bm}
\usepackage{verbatim}
\usepackage{todonotes}
\usepackage[
	colorlinks,
	linkcolor={black},
	citecolor={black},
	urlcolor={black}
]{hyperref}
\usepackage{cite}

\newcommand{\bbC}{{\mathbbm{C}}}
\newcommand{\bbD}{{\mathbbm{D}}}
\newcommand{\bbN}{{\mathbbm{N}}}

\newcommand{\bbR}{{\mathbbm{R}}}

\newcommand{\bbZ}{{\mathbbm{Z}}}

\newcommand{\calE}{{\mathcal{E}}}

\newcommand{\calZ}{{\mathcal{Z}}}

\newcommand{\cJ}{{\mathcal{J}}}

\newcommand{\rmd}{{\mathrm{d}}}

\newcommand{\SL}{{\mathrm{SL}}}
\newcommand{\SU}{{\mathrm{SU}}}

\newcommand{\spr}{{\mathrm{spr}}}
\newcommand{\tr}{{\mathrm{Tr}}}

\newcommand{\inv}{^{-1}}
\renewcommand{\Im}{\operatorname{Im}}
\renewcommand{\Re}{\operatorname{Re}}
\newcommand{\AC}{\operatorname{AC}}
\newcommand{\LP}{{\mathrm{LP}}}

\newcommand{\Leb}{{\mathrm{Leb}}}

\newcommand\scalemath[2]{\scalebox{#1}{\mbox{\ensuremath{\displaystyle #2}}}}

  \newcommand{\loc}{\mathrm{loc}}
\newcommand{\iop}{{\mathrm{i}}}
\newcommand{\eop}{{\mathrm{e}}}
  \DeclareMathOperator\supp{supp}

\allowdisplaybreaks

\newtheorem{theorem}{Theorem}[section]
\newtheorem{prop}[theorem]{Proposition}
\newtheorem{coro}[theorem]{Corollary}
\newtheorem{lemma}[theorem]{Lemma}

\theoremstyle{definition}
\newtheorem{definition}[theorem]{Definition}

\definecolor{purple}{rgb}{.5,0,1}
\definecolor{orange}{rgb}{1,.5,0}
\definecolor{green}{rgb}{0,.4,0}

\def\set#1{\left\{#1\right\}}

\numberwithin{equation}{section}

\author[B.\ Eichinger]{Benjamin Eichinger}
\thanks{B.E. was spported by the Austrian Science Fund FWF, project no: P33885}
\address{Institute of Analysis and Scientifc Computing, Vienna University of Technology, Wien A-1040, Austria}
\email{benjamin.eichinger@tuwien.ac.at}

\author[J.\ Fillman]{Jake Fillman}
\thanks{J.F.\ was supported in part by Simons Foundation Collaboration Grant \#711663.}
\address{Department of Mathematics, Texas State University, San Marcos, TX 78666, USA}
\email{fillman@txstate.edu}

\author[E.\ Gwaltney]{Ethan Gwaltney}
\thanks{E.G.\ was supported in part by NSF grant DMS--1745670.}
\address{Department of Mathematics, Rice University, Houston, TX~77005, USA}
\email{ethan.gwaltney@rice.edu}

\author[M.\ Luki\'{c}]{Milivoje Luki\'{c}}
\address{Department of Mathematics, Rice University, Houston, TX~77005, USA}
\email{milivoje.lukic@rice.edu}
\thanks{M.L.\ was supported in part by NSF grant DMS--1700179.}

\title[Dirac Operators with Thin Spectra]{Limit-Periodic Dirac Operators with Thin Spectra}

\begin{document}
\maketitle

\begin{abstract} We prove that limit-periodic Dirac operators generically have spectra of zero Lebesgue measure and that a dense set of them have spectra of zero Hausdorff dimension. The proof combines ideas of Avila from a Schr\"odinger setting with a new commutation argument for generating open spectral gaps. This overcomes an obstacle previously observed in the literature; namely, in Schr\"odinger-type settings, translation of the spectral measure corresponds to small $L^\infty$-perturbations of the operator data, but this is not true for Dirac or CMV operators. The new argument is much more model-independent. To demonstrate this, we also apply the argument to prove generic zero-measure spectrum for CMV matrices with limit-periodic Verblunsky coefficients. 
 \end{abstract}

\setcounter{tocdepth}{1}

\tableofcontents

\hypersetup{
	linkcolor={black!30!blue},
	citecolor={black!30!blue},
	urlcolor={black!30!blue}
}

\section{Introduction} \label{sec:intro}

We study Dirac operators in the form
\[
	\Lambda_\varphi =  \begin{bmatrix}
	\iop & 0 \\ 0 & -\iop
\end{bmatrix}	 \frac{{\rmd}}{{\rmd}x} + \begin{bmatrix} 0 & \varphi(x) \\ \overline{\varphi(x)} & 0 \end{bmatrix}
\]
with operator data $\varphi:\bbR \to \bbC$; 
up to a pointwise unitary conjugation, this is equivalent to the classical form of Dirac operators \cite{LevitanSargsjan, GrebertKappeler2014, ClarkGesztesy} given by
\[
L_\varphi =  \begin{bmatrix}
	0 & -1 \\ 1 & 0
\end{bmatrix}	 \frac{{\rmd}}{{\rmd}x} - \begin{bmatrix} \Re \varphi(x) & \Im \varphi(x) \\  \Im \varphi(x) &  -\Re \varphi(x) \end{bmatrix}.
\]
We assume $\varphi \in L^\infty(\bbR)$; in this case, $\Lambda_\varphi$ is an unbounded self-adjoint operator on $L^2(\bbR,\bbC^2)$ with domain $H^1(\bbR,\bbC^2)$.

As one of the simplest classes of differential operators with equal deficiency indices (and therefore existence of self-adjoint operators), Dirac operators have historically been studied in parallel with Schr\"odinger operators. The later discoveries of integrable PDEs further motivate the importance of the corresponding classes of operators; in particular, just as Schr\"odinger operators appear in the Lax pair representation of the KdV equation, Dirac operators of the form $\Lambda_\varphi$ appear in the Zakharov--Shabat Lax pair representation for the defocusing NLS \cite{ZakharovShabat, GrebertKappeler2014}.  The results in this paper can also be motivated from this point of view:  the inverse spectral theory of reflectionless Schr\"odinger operators \cite{SodinYuditskii2,GesztesyYuditskii} and the study of almost periodicity in time of solutions of the KdV equation with almost periodic data \cite{BDGL, Egorova, Egorova2, EVY} require ``thickness'' assumptions on the spectrum such as the Widom condition, the ``direct Cauchy theorem'' property and a gap summability condition; thus, constructions of Schr\"odinger operators with thin spectra \cite{DFL2017JST, DFG2014AHP, LenzSeifStol2014JDE} indicate their limitations. Analogously, the results of this paper indicate the limitations of existing inverse results about reflectionless Dirac operators \cite{EgorovaSurkova,BLY2} and the defocusing NLS equation with almost periodic data.

 We say that $\varphi$ is \emph{periodic} (of period $T>0$) if $\varphi=\varphi(\cdot - T)$. We say that $\varphi$ is (uniformly) \emph{limit-periodic} if it lies in the closure of the set of periodic elements of $C(\bbR)$ (in the $L^\infty$ topology).

Let $\LP(\bbR,\bbC)$ denote the set of all limit-periodic functions $\bbR \to \bbC$.
This is a complete metric space in the $L^\infty$ metric (note however that $\LP(\bbR,\bbC)$ is not a Banach space since the sum of periodic functions with incommensurable frequencies will not be limit-periodic in general).

Recall that a residual subset of a complete metric space $X$ is one that contains a dense $G_\delta$ subset of $X$. We say that a property holds for \emph{generic} $x \in X$ if the set of $x$ for which it holds is residual.

We say that $S \subseteq \bbR$ is a \emph{generalized} Cantor set if it is closed (not necessarily compact), perfect, and nowhere dense. Our main result is that the spectra of Dirac operators with limit-periodic potentials are (typically) generalized Cantor sets that are moreover very thin in the measure-theoretic sense.

\begin{theorem} \label{t:dirac:zeromeas}
For generic $\varphi \in \LP(\bbR,\bbC)$, $\sigma(\Lambda_\varphi)$ is a generalized Cantor set of zero Lebesgue measure, and the spectral type of $\Lambda_\varphi$ is purely singular continuous.
\end{theorem}

One can strengthen the zero-measure statement to spectrum of zero Hausdorff dimension for a dense set of operator data:

\begin{theorem} \label{t:dirac:zerohd}
For a dense set of $\varphi \in \LP(\bbR,\bbC)$, $\sigma(\Lambda_\varphi)$ is a generalized Cantor set of zero Hausdorff dimension and zero lower box-counting dimension, and the spectral type of $\Lambda_\varphi$ is purely singular continuous.
\end{theorem}

Let us comment on the proofs of Theorems~\ref{t:dirac:zeromeas} and \ref{t:dirac:zerohd}. Beginning with the seminal paper of Avila \cite{Avila2009CMP}, there is by now a well-established path to obtaining thin spectra for limit-periodic operators, provided one can perform a version of Avila's perturb-and-grow technique \cite{DFL2017JST, DFW2021preprint, DamanikGan2010JFA,  FillmanOng2017JFA}. The construction of \cite{Avila2009CMP} begins with a periodic operator, performs a finite number of small perturbations to move energies out of the spectrum, and exploits uniform hyperbolicity of cocycles in the resolvent set in conjunction with connections between the density of states and rotation number. The key perturbative argument is done in two steps: first, to open up many small gaps in the spectrum via a perturbation to an operator of much higher period, and then to shift these new gaps around in a carefully controlled fashion. 

The first step is generally straightforward to implement, and follows readily from Floquet theory; compare \cite{Simon1976AIHP}. The second step has traditionally relied on small translations (or dilations \cite{DFW2021preprint}) of the spectral measure in the self-adjoint setting, or rotations of the spectral measure in the unitary setting. In the Schr\"odinger setting, a translation of the spectral measure corresponds to a constant shift to the potential, so small translations correspond to uniformly small perturbations of the potential. This is not the case for Dirac operators: translation of the spectral measure corresponds to multiplication of the operator data $\varphi(x)$ by $\eop^{\iop kx}$, which is in general not a small perturbation in $L^\infty(\bbR)$! A similar obstruction was noted in \cite{FillmanOng2017JFA} in the setting of CMV matrices: rotation by angle $\theta$ corresponds to multiplication of the $n$-th Verblunsky coefficient $\alpha_n$ by $\eop^{-\iop (n+1)\theta}$ (cf.\ \cite[p.\ 960]{Simon2005OPUC2}).
In that paper, the authors noted that defect and overcame it by enlarging the class of operators under consideration to include simple spectral shifts. However, they noted at the time that this enlargement of the space of operators was somewhat contrived and that one should be able to perform the desired perturbative analysis without passing to an artificial enlargement of the parameter space, stating in particular ``\textit{Additional ideas are needed to refine our techniques down to this setting; we regard this as an interesting open question}.'' \cite[Page~5114]{FillmanOng2017JFA}.

We overcome this difficulty in Section~\ref{sec:proofs} by using an indirect argument to move energies out of the spectrum via noncommutation of transfer matrices, which is itself inspired by recent work on verifying the hypotheses of Furstenberg's theorem \cite{Furstenberg1963TAMS} via ideas in inverse spectral theory \cite{BDFGVWZ2019JFA}. We also exploit compactness to simplify some arguments in a manner that has not been exploited in the current setting before.

The approach described for Dirac operators is robust in the sense that it can be applied in any situation in which one has suitable inverse spectral results. To demonstrate the versatility of this approach, we prove in Section~\ref{sec:OPUC} related theorems for extended CMV matrices.

Let $\bbD = \{z \in \bbC : |z| <1\}$ and $\partial \bbD = \{z \in \bbC : |z|=1\}$ denote the open unit disk and the unit circle in $\bbC$. Given a sequence $\alpha = \{\alpha_n\}_{n\in\bbZ}\in\bbD^\bbZ$, the corresponding \emph{extended CMV matrix} $\mathcal{E}=\mathcal{E}_\alpha$ is given by
\begin{equation} \label{def:extcmv}
\scalemath{.75}{\mathcal{E}
=
\begin{bmatrix}
\ddots & \ddots & \ddots &&&&&  \\
\overline{\alpha_0}\rho_{-1} & -\overline{\alpha_0}\alpha_{-1} & \overline{\alpha_1}\rho_0 & \rho_1\rho_0 &&& & \\
\rho_0\rho_{-1} & -\rho_0\alpha_{-1} & -\overline{\alpha_1}\alpha_0 & -\rho_1 \alpha_0 &&& & \\
&  & \overline{\alpha_2}\rho_1 & -\overline{\alpha_2}\alpha_1 & \overline{\alpha_3} \rho_2 & \rho_3\rho_2 & & \\
& & \rho_2\rho_1 & -\rho_2\alpha_1 & -\overline{\alpha_3}\alpha_2 & -\rho_3\alpha_2 &  &  \\
& &&& \overline{\alpha_4} \rho_3 & -\overline{\alpha_4}\alpha_3 & \overline{\alpha_5}\rho_4 & \rho_5\rho_4 \\
& &&& \rho_4\rho_3 & -\rho_4\alpha_3 & -\overline{\alpha_5}\alpha_4 & -\rho_5 \alpha_4  \\
& &&&& \ddots & \ddots &  \ddots
\end{bmatrix}},
\end{equation}
where $\rho_n = (1-|\alpha_n|^2)^{1/2}$. 

The extended CMV operator is a significant object in mathematical physics, with connections to orthogonal polynomials \cite{Simon2005OPUC1, Simon2005OPUC2}, quantum walks on the integers \cite{CGMV2010CPAM, CGMV2012QIP}, and gap-labelling problems for the ferromagnetic Ising model \cite{DFLY2015IMRN, DamMunYes2013JSP}.

Naturally, $\alpha$ is $q$-\emph{periodic} for $q \in \bbN$ if $\alpha_{n+q}\equiv \alpha_n$. To avoid trivialities in the present setting, we only want to consider $\alpha$ that are bounded way from $\partial \bbD$, that is $\|\alpha\|_\infty < 1$. On the other hand, in order to apply Baire category arguments, one wants to work with a complete metric space of operator data. This was achieved in \cite{FillmanOng2017JFA} by fixing an \emph{a priori} bound $0<r<1$ and considering those limit-periodic $\alpha$ for which $\|\alpha\|_\infty \leq r$. The following definition gives us a way to consider all limit-periodic sequences in $\bbD$ that are bounded away from $\partial \bbD$ without enforcing \emph{a priori} bounds.

\begin{definition} \label{def:poincaremetric}
Equip $\bbD$ with the Poincar\'{e} metric $\delta(z_1,z_2) = \mathrm{tanh}^{-1}\left|{\frac  {z_1-z_2}{1-z_1\overline{z_2}}}\right|$. For sequences $\alpha,\beta \in \bbD^\bbZ$, denote by
\begin{equation} \label{eq:CMV:inftyHyperbolicDef} \delta(\alpha,\beta) = \sup\{ \delta(\alpha_n,\beta_n) : n \in \bbZ \} 
\end{equation}
the induced metric on $\bbD^\bbZ$.
\end{definition}

Notice that $\delta(\alpha,0) < \infty $ if and only if $\sup_n |\alpha_n| < 1$. Let us say that $\alpha \in \bbD^\bbZ$ is \emph{limit-periodic} if there exist periodic sequences $\alpha^{(n)} \in \bbD^\bbZ$ such that $\delta(\alpha^{(n)} , \alpha) \to 0$.  Denote by $\LP(\bbZ,\bbD)$ the set of limit-periodic sequences. The reader can readily check that $\LP(\bbZ,\bbD)$ is complete in the metric $\delta$ and that $\sup_n |\alpha_n|<1$ for every $\alpha \in \LP(\bbZ,\bbD)$.

Since extended CMV operators are unitary, their spectra are contained in $\partial \bbD$, the unit circle. We will say that $S \subseteq \partial \bbD$ is a Cantor subset of $\partial \bbD$ if it is closed, perfect, and nowhere dense (in the relative topology as a subset of $\partial \bbD$).

\begin{theorem} \label{t:extCMV:zeromeas}
For generic $\alpha \in \LP(\bbZ,\bbD)$, $\sigma(\calE_\alpha)$ is a Cantor subset of $\partial \bbD$ of zero Lebesgue measure, and the spectral type  of $\calE_\alpha$ is purely singular continuous. \end{theorem}

As in the Dirac case, one can strengthen the result to show the spectrum has zero Hausdorff dimension for a dense set of operator data:

\begin{theorem}\label{t:extCMV:zerohd}
For a dense set of $\alpha \in \LP(\bbZ,\bbD)$, $\sigma(\calE_\alpha)$ is a Cantor subset of $\partial \bbD$ of zero Hausdorff dimension and zero lower box-counting dimension, and the spectral type of $\calE_\alpha$ is purely singular continuous.
\end{theorem}

The paper is organized as follows. We recall some general facts about Dirac operators in Section~\ref{sec:dirac}. We prove Theorems~\ref{t:dirac:zeromeas} and \ref{t:dirac:zerohd} in Section~\ref{sec:proofs} and Theorems~\ref{t:extCMV:zeromeas} and \ref{t:extCMV:zerohd} in Section~\ref{sec:OPUC}. 

\subsection*{Acknowledgements} We are grateful to Christian Sadel and Hermann Shulz-Baldes for helpful conversations. J.F.\ thanks the American Institute of Mathematics for hospitality and support during a January 2022 visit, during which part of this work was completed.

\section{Preparatory Work for Dirac Operators} \label{sec:dirac}

\subsection{Generalities}
In this section, we consider various properties of Dirac operators on $\bbR$ that are essential in the proofs of Theorems~\ref{t:dirac:zeromeas} and \ref{t:dirac:zerohd} (see also \cite{LevitanSargsjan,GrebertKappeler2014,ClarkGesztesy,ClarkGesztesy2}).

We will denote
\begin{equation} \label{eq:dirac:variousJs}
	j = \begin{bmatrix} -1 & 0 \\ 0 & 1 \end{bmatrix}, \qquad  J = \begin{bmatrix} 0 & \iop  \\ -\iop  & 0 \end{bmatrix}, \qquad  \cJ =  \begin{bmatrix} 0 & 1 \\ 1 & 0 \end{bmatrix}
\end{equation}
Note that $\cJ, -J,$ and $-j$ are the Pauli matrices, often denoted $\sigma_1, \sigma_2,$ and $\sigma_3,$ respectively. The Dirac operator $\Lambda_\varphi$ with operator data $\varphi:I \to \bbC$ is defined by the differential expression 
\[
\Lambda_\varphi = -\iop j\frac{{\rmd}}{{\rmd}x} + \Phi(x), \qquad
 \Phi(x) = \begin{bmatrix} 0 & \varphi(x) \\ \overline{\varphi(x)} & 0 \end{bmatrix},
\]
together with an appropriately defined domain in the Hilbert space $L^2(I, \bbC^2)$. On the line $I = \bbR$, if $\varphi$ is uniformly locally $L^2$ in the sense that
\begin{equation}\label{UniformlyLocallyL2}
\sup_{x\in\bbR} \int_x^{x+1} \lvert \varphi(t) \rvert^2 \, {\rmd}t < \infty,
\end{equation}
then the operator is limit-point at $\pm \infty$ and  $\Lambda_\varphi$ defines a self-adjoint operator on the domain $D(\Lambda_\varphi) = H^1(\bbR, \bbC^2)$. We employ a standard abuse of notation here, writing $\Lambda_\varphi$ both for the self-adjoint operator on $D(\Lambda_\varphi)$ and the differential expression, which may act on any function with at least one (weak) derivative.

For any $z \in \bbC$, the Dirac eigenequation on an interval $I \subseteq \bbR$ is given by 
\begin{align} \label{eq:general:eigenEq}
	\Lambda_\varphi U(x,z) = z U(x,z), \quad U(\cdot,z) \in \AC_\loc(I,\bbC^2),
\end{align}
where $\AC_\loc(I,\bbC^2)$ denotes those functions that are absolutely continuous on compact subintervals of $I$. A solution $U$ of \eqref{eq:general:eigenEq} is called an \emph{eigensolution} at $z$. The \emph{Wronskian} of any two functions $U,V \in \AC_\loc(I, \bbC^2)$ is defined by
\[
W[U,V](x) = U(x)^\top JV(x) = \iop (U_1(x)V_2(x)-U_2(x)V_1(x)), \quad x \in I.
\]
The limit-point conditions also state that the boundary Wronskian at $\pm\infty$ is trivial, i.e.\ for all $U, V\in D(\Lambda_\varphi)$, 
\[
\lim_{x\to \pm \infty} W[U,V](x) = 0.
\]
By computing (using $\cJ = \iop jJ$)
\begin{equation} \label{eq:dirac:wronskConserve}
	W[U,V]'(x) = (\Lambda_\varphi U(x))^\top  \cJ V(x) - U(x)^\top   \cJ  (\Lambda_\varphi V(x)),
\end{equation}
we can see that if $U$ and $V$ are eigensolutions at a given $z \in \bbC$, $W[U,V](x)$ is a constant independent of $x$. Moreover, $W[U,V]=0$ if and only if $U$ and $V$ are linearly dependent.

A \emph{Weyl solution} at $z \in \bbC$ for the endpoint $\pm\infty$ is a nontrivial eigensolution $\psi^\pm(x,z)$ at $z$ that is square-integrable on the half-line $[0,\pm \infty)$; due to \eqref{UniformlyLocallyL2}, there is a unique (up to normalization) Weyl solution at each endpoint $\pm\infty$ for every $z\in \bbC \setminus \sigma(\Lambda_\varphi)$, and $W[\psi^+, \psi^-] \neq 0$. We will write formulas in a normalization-independent way, unless a normalization for $\psi^\pm$ is explicitly stated.

Weyl functions generate the \emph{Green function} via
\begin{equation} \label{eq:generaldirac:greenFctFromWeyl}
G(x,y;z,\varphi) = 
	\begin{cases}
		\frac{1}{W[\psi^+,\psi^-]} \psi^-(x,z) \psi^+(y,z)^\top \mathcal{J}, \ x < y \\
		\frac{1}{W[\psi^+, \psi^-]} \psi^+(x,z) \psi^-(y,z)^\top \mathcal{J}, \ x > y
	\end{cases}
\end{equation}
which is the integral kernel of $(\Lambda_\varphi - z)^{-1}$ in the sense that
\begin{equation} \label{eq:generaldirac:greenFctIntKer}
((\Lambda_\varphi - z)^{-1} f )(x) = \int G(x,y;z,\varphi) f(y) \, {\rmd}y, \qquad \forall f\in D(\Lambda_\varphi),
\end{equation}
which the reader can check by a direct calculation by differentiating under the integral sign.

\noindent Given $x\in\bbR$ and $z \in \bbC$, let $A_z(x,\varphi)$ be the matrix solution of 
\begin{align}\label{eq:EV}
\Lambda_\varphi A_z(x,\varphi) = z A_z(x,\varphi),\quad A_z(0,\varphi)= I.
\end{align}
These are transfer matrices starting from $0$; defining  the matrices
\begin{equation} \label{eq:dirac:A_zyxDef}
A_z(y,x,\varphi)=A_z(y,\varphi)A_z(x,\varphi)^{-1}
\end{equation}
for $x,y \in \bbR$, $z \in \bbC$, one has
\begin{equation}
U(y) = A_z(y,x,\varphi)U(x)
\end{equation}
whenever $U$ is an eigensolution of $\Lambda_\varphi$ at $z$. 

Notice that conservation of the Wronskian as in \eqref{eq:dirac:wronskConserve} implies $\det(A_z(y,x,\varphi)) = 1$ for all $x$, $y$, $z$, and $\varphi$. For $z, w\in\bbC$, differentiating $A_w(y,x,\varphi)^*jA_z(y,x,\varphi)$ with respect to $y$ and using \eqref{eq:EV} we get 
\begin{equation} \label{eq:dirac:A*jAderiv}
\partial_y(A_w(y,x,\varphi)^*jA_z(y,x,\varphi)) = \iop (z-\overline{w})A_w(y,x,\varphi)^* A_z(y,x,\varphi)
\end{equation}
which (using $A_w(x,x,\varphi) = A_z(x,x,\varphi) = I$) leads to
\begin{equation} \label{eq:dirac:A*jA}
A_w(y,x,\varphi)^*jA_z(y,x,\varphi)-j
= \iop (z-\overline{w})\int_x^yA_w(t,x,\varphi)^*A_z(t,x,\varphi) \, {\rmd}t.
\end{equation}
In particular, applying \eqref{eq:dirac:A*jA} at a real parameter $z =w = \lambda \in \bbR$, one has
\[
A_\lambda(y,x,\varphi)^*jA_\lambda(y,x,\varphi) = j,
\]
which (together with $\det A_\lambda = 1$) implies $A_\lambda(y,x,\varphi) \in \SU(1,1)$.

The  Schur functions associated to $\varphi$ at the point $x \in \bbR$ are defined using the Weyl solutions as
\begin{equation}
s_+(x,z) = \frac{ \psi^+_1(x,z)} {\psi^+_2(x,z)}, \qquad s_-(x,z) = \frac{ \psi^-_2(x,z)} {\psi^-_1(x,z)}.
\end{equation}
For each fixed $x \in \bbR$, $s_\pm(x,\cdot)$ is an analytic function from the upper half-plane $\bbC_+ = \{z : \Im z>0\}$ to $\bbD$ (compare \cite[Lemma~2.1]{EGL} and surrounding discussion). These can be characterized by a Weyl disk formalism: Weyl disks in this setting can be defined in $\hat{\bbC} = \bbC \cup\{\infty\}$  as
\[
D(x,z) = \left\{ w \in \hat\bbC : \begin{bmatrix} w \\ 1 \end{bmatrix}^* A_z(x)^* j A_z(x) \begin{bmatrix} w \\ 1 \end{bmatrix} \ge 0 \right\},
\]
with the natural convention $[\infty,1] =[1,0]$ in projective coordinates. The relation \eqref{eq:dirac:A*jA} ensures their nesting property, $D(x_2, z) \subseteq D(x_1,z)$ whenever $x_1 < x_2$. We are in the limit-point case so one has
\begin{equation}
\{ s_+(0,z) \} = \bigcap_{x \ge 0} D(x,z).
\end{equation}

We will now discuss some foundational results with proofs, including a Combes--Thomas estimate and a Schnol's theorem in the Dirac setting. We formulate some results for $\psi^+$, but analogous results hold for $\psi^-$.

\begin{lemma} \label{lemmaNontrivialMultiple}
Suppose $\varphi:\bbR \to \bbC$ obeys \eqref{UniformlyLocallyL2},  let $\psi^+(x,z)$ be a Weyl solution at $+\infty$, and suppose $z \notin \sigma(\Lambda_\varphi)$. For any $[c,d] \subseteq \bbR$, there exists $f \in L^2(\bbR, \bbC^2)$ such that $f \chi_{[c,d]} = f$ and $(\Lambda_\varphi - z)\inv f$ is a nontrivial constant multiple of $\psi^+$ on $[d,\infty)$.
\end{lemma}

\begin{proof}
If $f \in L^2(\bbR,\bbC^2)$ with $f \chi_{[c,d]} = f$, then \eqref{eq:generaldirac:greenFctFromWeyl} and \eqref{eq:generaldirac:greenFctIntKer} yield the following for $x > d$:
\begin{align*}
((\Lambda_\varphi - z)^{-1} f )(x)
&  = \int G(x,y;z,\varphi) f(y) \, {\rmd}y\\ 
& =   \frac 1{ W[\psi^+, \psi^-] } \int  \psi^+(x, z) \psi^-( y,z)^\top \cJ f(y) \, {\rmd}y. \\ 
& = C \psi^+(x,z),
\end{align*}
where $C = C(f) = \frac{1}{W} \int_c^d \psi^{-}(y,z)^\top \mathcal{J} f(y) \, \rmd y$. Since $\psi^-$ does not vanish identically, $f \in L^2([c,d], \bbC^2)$ can be chosen so that $C(f)$ is nonzero.
\end{proof}

The proof applies more generally, for any kind of left endpoint; in particular, the same proof yields the following:

\begin{lemma} \label{lemmaNontrivialMultiple2}
Suppose $\varphi: [0,\infty) \to \bbC$ obeys
\[
\sup_{x \ge 0} \int_x^{x+1} \lvert \varphi(t) \rvert^2 \, {\rmd}t < \infty,
\]
 let $\psi^+(x,z)$ be a Weyl solution at $+\infty$, and let $\Lambda_\varphi$ denote the Dirac operator on $[0,\infty)$ with a Dirichlet boundary condition $f_1(0) = f_2(0)$ at zero. Suppose $z \notin \sigma(\Lambda_\varphi)$. For any $[c,d] \subseteq [0,\infty)$, there exists $f \in L^2( [0,\infty), \bbC^2)$ such that $f \chi_{[c,d]} = f$ and $(\Lambda_\varphi - z)\inv f$ is a nontrivial constant multiple of $\psi^+$ on $[d,\infty)$.
\end{lemma}

\begin{lemma}
\label{Dirac.Weyl.exp.decays.lemma}
If $\varphi$ obeys \eqref{UniformlyLocallyL2}, then for any $z  \in \bbC \setminus \sigma_{\rm{ess}}(\Lambda_\varphi)$, the Weyl solution $\psi^+(x,z )$ obeys $\|\psi^+(x,z )\| = O(\eop^{-\gamma x})$ as $x \to +\infty$, for some $\gamma > 0$.
\end{lemma}

\begin{proof}
Since $\psi^+$ is a nontrivial eigensolution, choose $\omega \in \partial\bbD$ such that
\begin{equation} \label{eq:diracgen:weyldecayOmegaChoice}
\begin{bmatrix} 1 & \omega
\end{bmatrix} \psi^+(0,z ) \neq 0
\end{equation}
(indeed, equality can hold for at most one value of $\omega$). Fix such an $\omega$ and define a half-line Dirac operator $\Lambda_0$ by the same differential expression as $\Lambda$ on the domain
\[
D(\Lambda_0) = \set{ f \in H^1( [0,\infty), \bbC^2 ) : \begin{bmatrix} 1 & \omega \end{bmatrix} f(0) = 0 }.
\]
Due to \eqref{eq:diracgen:weyldecayOmegaChoice} and the assumption $z  \notin \sigma_{\rm ess}(\Lambda_\varphi)$, we have $z  \notin \sigma(\Lambda_0)$. Given $\gamma>0$, define $\Lambda_\gamma$ as $\Lambda_\gamma := \eop^{\gamma x} \Lambda_0 \eop^{-\gamma x}$ with domain $D(\Lambda_\gamma) = D(\Lambda_0)$.
 To see that $\Lambda_\gamma$ with such a domain defines a (non-self adjoint) operator, note that the difference  
\begin{equation}
	\Lambda_\gamma - \Lambda_0 = \iop \gamma j,
\end{equation}
is a bounded (non-self adjoint) operator. 

If $V \in D(\Lambda_\gamma)$, then using $z  \notin \sigma(\Lambda_0)$ and \mbox{$\|V\| \leq \|(\Lambda_0 - z )\inv\| \|(\Lambda_0 - z )V\|$} yields  
\[
	\|(\Lambda_\gamma - \Lambda_0)V\| = \|\iop \gamma j V\| =\gamma \|V\| \leq C\gamma \|(\Lambda_0 - z )V\|
\]
with $C =\|(\Lambda_0-z )^{-1}\|  > 0$. Taking $\gamma >0$ sufficiently small, we obtain 
\[
	\|(\Lambda_\gamma - \Lambda_0)(\Lambda_0 - z )\inv\| < 1,
\]
so that 
\[
	\Lambda_\gamma - z  = [(\Lambda_\gamma - \Lambda_0)(\Lambda_0 - z )\inv + I](\Lambda_0 - z )
\]
is invertible. By Lemma~\ref{lemmaNontrivialMultiple2}, there exists $f \in L^2([0,\infty); \bbC^2)$ such that $\supp f \subseteq [0,1]$ and \mbox{$(\Lambda_0 - z )\inv f$} is a nontrivial multiple of $\psi^+$ on $[1,\infty)$. Then, on $[1,\infty)$,
\[
	(\Lambda_\gamma - z )\inv (\eop^{\gamma x}f) 
\]
is a nontrivial multiple of $\eop^{\gamma x}\psi$ on $[1,\infty)$. Since $(\Lambda_\gamma - z )\inv (\eop^{\gamma x}f) \in D(\Lambda_0) \subseteq L^\infty([0,\infty); \bbC^2)$ by a Sobolev embedding theorem, $\|\psi^+(x,z )\| = O(\eop^{-\gamma x})$ as $x \to \infty$.
\end{proof}

\begin{theorem}[Schnol's Theorem] \label{t:schnolForDirac}
Let $\Lambda = \Lambda_\varphi$ be a Dirac operator $\varphi$ satisfying \eqref{UniformlyLocallyL2}. For a fixed $\kappa > 1/2$, let $S_\kappa$ denotes the set of $\lambda \in \bbC$ for which there exists a nontrivial eigensolution $U(x,\lambda)$ obeying $\|U(x,\lambda)\| = O(|x|^\kappa)$ as $x \to \pm \infty$. Then:
\begin{enumerate}[label={\rm(\alph*)}]
	\item $S_\kappa \subseteq \sigma(\Lambda)$;
	\smallskip
	
	\item The maximal spectral measure of $\Lambda$ is supported on $S_\kappa$;
	\smallskip
	
	\item $\overline{S_\kappa} = \sigma(\Lambda)$.
\end{enumerate}
\end{theorem}
\begin{proof}
(a)
Let $\lambda \notin \sigma(\Lambda)$ and let $\psi^{\pm}(x,\lambda)$ be the associated Weyl solutions at $\pm \infty$. By Lemma \ref{Dirac.Weyl.exp.decays.lemma}, the Weyl solutions obey 
$$\|\psi^{\pm}(x,\lambda)\| = O(\eop^{-\gamma |x|}) \text{ as }x \to \pm\infty,$$
for some $\gamma > 0$. Suppose for the purpose of establishing a contradiction that $\lambda \in S_\kappa$, and let $U(x,\lambda)$ denote an eigensolution satisfying $\|U(x,\lambda)\| = O(|x|^\kappa)$ as $x \to \pm\infty$. The growth estimates on $U$ and $\psi^\pm$ imply that the Wronskian satisfies $W[U,\psi^\pm](x) = O(\eop^{-\gamma |x|}|x|^\kappa)$ as $x \to \pm\infty$. Using conservation of the Wronskian, we have $W[U,\psi^\pm] = 0$, which implies that $\psi^+$ and $\psi^-$ are linearly dependent, which in turn implies $\lambda \in \sigma(\Lambda)$, a contradiction.
\medskip

(b) The eigenfunction expansion for $\Lambda$ is a unitary map $\mathcal{U}: L^2(\mathbb{R}, \mathbb{C}^2; dx) \to L^2(\mathbb{R}, \mathbb{C}^2; d\Omega)$, with $d\Omega(\lambda) = W(\lambda) \, d\mu(\lambda)$ a matrix-valued measure and $\mu$ a maximal spectral measure for $\Lambda$, defined on compactly supported functions by 
\begin{equation}
(\mathcal{U}f)(\lambda) = \int A_\lambda(x)^*  f(x) \, {\rmd}x
\end{equation}
and extended by continuity.

 For fixed $x \in\mathbb{R},$ $z \in \mathbb{C} \setminus \mathbb{R}$, and $e_1 := (1,0)^\top$, $e_2 := (0,1)^\top,$ the functions $f_k(y) = G(x,y; z)e_k$ are mapped to $(\mathcal{U}f_k)(\lambda) = \frac{1}{\lambda - z} A_\lambda(y)^* e_k$. Since $\mathcal{U}$ is unitary, 
	\begin{equation}
		\int \|G(x,y;z)e_k\|^2 \, {\rmd}x 
		= \int e_k^* A_\lambda(y) W(\lambda)A_\lambda(y)^*  e_k |\lambda - z|^{-2} \, {\rmd}\mu(\lambda). \label{squared.norms.identity}
	\end{equation}
	Now, since $D(\Lambda) \subseteq L^\infty(\mathbb{R}, \mathbb{C}^2)$, for any $f \in L^2(\mathbb{R}, \mathbb{C}^2)$, 
	\[
		\sup_{x \in \mathbb{R}} \left|\int e_k^* G(x,y; z)^* f(y) \, {\rmd}y \right| < \infty,
	\]
	so by the uniform boundedness principle, 
	\[
		\sup_{x \in \mathbb{R}} \int \|G(x,y;z)e_k\|^2 \, {\rmd}y < \infty.
	\]
	This combined with \eqref{squared.norms.identity} implies
	\[
		\sup_{x \in \mathbb{R}} \int e_k^*A_\lambda(y)W(\lambda)A_\lambda(y)^*e_k |\lambda - z|^{-2} \, {\rmd}\mu(\lambda) < \infty.
	\]
	Multiplying by the integrable function $(1 + x^2)^\kappa$ and integrating in $x$ yields
	\[
		\iint (1 + x^2)^\kappa e_k^*A_\lambda(y)W(\lambda)A_\lambda(y)^*e_k |\lambda - z|^{-2} \, {\rmd}\mu(\lambda) \,  {\rmd}x < \infty.
	\]
	Fubini's theorem then implies
	\begin{equation}
		\int (1 + x^2)^\kappa e_k^*A_\lambda(y)W(\lambda)A_\lambda(y)^*e_k \, {\rmd}x < \infty, \ \mu\text{-a.e.}\ \lambda, \ k = 1, 2. \label{mu.almost.everywhere.finiteness}
	\end{equation}
	Gronwall's inequality for $\Lambda_\varphi  A_z = z A_z$ gives an estimate
	\[
	\|A_\lambda(x)\| \leq C\int_{x}^{x+1} \|A_\lambda(y)\| \, {\rmd}y
	\]
	with an $x$-independent value of $C$. Thus, \eqref{mu.almost.everywhere.finiteness} implies $\|U(x,\lambda)\| = O(x^\kappa)$ for $\mu$-a.e. $\lambda$.
\medskip

(c)  $S_\kappa \subseteq \sigma(\Lambda)$ implies $\overline{S_\kappa} \subseteq \sigma(\Lambda)$; conversely, since $\mu$ is supported on $S_\kappa$, $\sigma(\Lambda) = \supp \mu \subseteq \overline{S_\kappa}$. 
\end{proof}

Let us briefly recall a tool that we will use later. The \emph{Hausdorff distance} between two nonempty closed subsets of $\bbR$ is given by
\begin{equation}
d_{\rm Hd}(K_1,K_2) = \inf \set{\varepsilon > 0 : K_1 \subseteq B_\varepsilon(K_2) \text{ and } K_2 \subseteq B_\varepsilon(K_1)},
\end{equation}
where $B_\varepsilon(S) = \{y : |x-y|<\varepsilon \text{ for some } x \in S\}$ denotes the open $\varepsilon$-neighborhood of the set $S$.
\begin{prop} \label{prop:dirac:HdSpecPert}
Given $\varphi_1,\varphi_2 \in L^\infty(\bbR,\bbC)$, one has
\begin{equation}
d_{\rm Hd}(\sigma(\Lambda_{\varphi_1}),\sigma(\Lambda_{\varphi_2}))
\leq \|\varphi_1 - \varphi_2\|_\infty.
\end{equation}
\end{prop}
\begin{proof}
This follows from \[d_{\rm Hd}(\sigma(\Lambda_{\varphi_1}),\sigma(\Lambda_{\varphi_2})) 
\leq \|\Lambda_{\varphi_1} - \Lambda_{\varphi_2}\| 
= \|\varphi_1 - \varphi_2\|_\infty. \]
The  inequality is a consequence of general perturbation theory for self-adjoint operators and the equality follows from a direct calculation.
\end{proof}

\subsection{Floquet Theory}
Now that we have set up the tools that we need for general Dirac operators with uniformly locally $L^2$ operator data, we specialize to the case of periodic $\varphi$. We assume that $\varphi \in C(\bbR,\bbC)$ has period $T>0$. Recall that $A_z(x,\varphi)$ and $A_z(y,x,\varphi)$ are the transfer matrices defined in \eqref{eq:EV} and \eqref{eq:dirac:A_zyxDef} respectively. Recall also from \eqref{eq:dirac:A*jA} that whenever $\lambda \in\bbR$, one has 
\begin{equation}
A_\lambda(y,x,\varphi)\in \SU(1,1) : = \{A \in \SL(2,\bbC) : A^*jA=j\},
\end{equation}
where $j$ is as in \eqref{eq:dirac:variousJs}.

 If $\varphi$ is periodic of period $T$, we call 
 \begin{equation} \label{eq:dirac:monodromyDef}  M_z(x,\varphi) := A_z(x+T,x,\varphi)
 \end{equation} the \emph{monodromy matrix} and $D(z,\varphi) = \tr(M_z(x,\varphi))$ the \emph{discriminant}. We sometimes abbreviate $M_z(\varphi) = M_z(0,\varphi)$ and write $M_z(x)$ when $\varphi$ is clear from context. Let us briefly summarize the aspects of Floquet theory that are needed for our proofs.

\begin{theorem}[Floquet Theory] \label{t:dirac:floquet}
Suppose $\varphi \in C(\bbR)$ is $T$-periodic and let $M_z(x,\varphi)$ and $D(z,\varphi) = \tr(M_z(x, \varphi))$ denote the associated monodromy and discriminant.
\begin{enumerate}
\item[{\rm (a)}] $\spr(M_z(x,\varphi))$ and $\tr(M_z(x,\varphi))$ do not depend on $x$. In particular, $D(z,\varphi)$ is well-defined.
\item[{\rm (b)}] The Lyapunov exponent\footnote{The first equality of \eqref{eq:dirac:L(z)Def} is the definition of $L(z)$.} is given by
\begin{equation} \label{eq:dirac:L(z)Def}
L(z) 
:= \lim_{x\to\infty} \frac{1}{x} \log\| A_z(x,0,\varphi)\| 
=\frac{1}{T} \log \spr(M_z(x,\varphi)).\end{equation}
\item[{\rm(c)}] $z$ is a generalized eigenvalue of $\Lambda_\varphi$  if and only if $D(z,\varphi) \in [-2,2]$.
\item[{\rm(d)}] The spectrum of $\Lambda_\varphi$ is given by
\[\Sigma = \sigma(\Lambda_\varphi) = \{\lambda \in \bbR : D(\lambda) \in [-2,2]\} = \calZ.\]
\item[{\rm(e)}] If $D(\lambda) \in (-2,2)$, then for each $x$ there is $B_\lambda(x) \in \SU(1,1)$ such that
\begin{equation} \label{eq:dirac:monodromyConjDef}
 B_\lambda(x)M_\lambda(x,\varphi)B_\lambda(x)^{-1} 
 \in \mathrm{K}= \set{ \begin{bmatrix}\eop^{\iop \theta} & 0 \\ 0 & \eop^{-\iop \theta} \end{bmatrix} : \theta \in [0,2\pi]}.
\end{equation}
This conjugacy is unique modulo left-multiplication by elements of $\mathrm{K}$. That is, if $B_\lambda^{(1)}(s)$ also satisfies \eqref{eq:dirac:monodromyConjDef}, one has $B_\lambda(s) = Q B_\lambda^{(1)}$ for some $Q \in \mathrm{K}$.
 \item[{\rm(f)}] Given $R>0$ there are at most $2\left(\frac{T}{\pi}(R + \|\varphi\|_\infty)+1\right)$ bands of $\sigma(\Lambda_\varphi)$ that intersect $[-R,R]$.

\end{enumerate}
\end{theorem}

\begin{proof} For the most part, these results are standard fare \cite{LevitanSargsjan}; we supply arguments to keep the paper more self-contained for the reader's convenience.

(a) This follows from periodicity and cyclicity of the trace.
\smallskip

(b) This follows from periodicity, Gelfand's formula, and interpolation.
\smallskip

(c) If $D(z) \in [-2,2]$, $M_z(\varphi)$ has a unimodular eigenvalue, so $\Lambda_\varphi U = zU$ enjoys a bounded solution, and thus $z$ is a generalized eigenvalue. If $D(z) \in \bbC \setminus [-2,2]$, then $M_z(\varphi)$ has eigenvalues $\lambda_\pm$ with $|\lambda_+|>1>|\lambda_-|$, so every solution of $\Lambda_\varphi U = zU $ grows exponentially on at least one half-line, and thus $z$ is not a generalized eigenvalue.
\smallskip

(d) Since $D(z)$ is analytic (in particular, continuous), the set of $z$ with $D(z) \notin [-2,2]$ is open, so this follows from (b), (c), and Theorem~\ref{t:schnolForDirac}.
\smallskip

(e) This a standard fact about $\SU(1,1)$ matrices. We will want to use some notation and ideas from the proof later, so we recall the argument here. Recall that any invertible $M \in \bbC^{2\times 2}$ acts on $\hat{\bbC}=\bbC \cup\{\infty\}$ by M\"obius transformations:
\[ \hat{M}z = \frac{m_{11}z+m_{12}}{m_{21}z+m_{22}}, \quad z \in \hat{\bbC}. \]
The reader can check that if $M \in \SU(1,1)$, then $\tr(M) \in (-2,2)$ if and only if $\hat{M}$ has a unique fixed point in $\bbD$. Furthermore, if $M \in \SU(1,1)$, then $\hat{M}0=0$ if and only if $M \in \mathrm{K}$, so the desired conjugacy is obtained by choosing $B \in \SU(1,1)$ with $\hat{B} \xi = 0$, where $\xi \in \bbD$ is the unique fixed point of $\hat{M}$ in the unit disk. The uniqueness statement follows since this discussion implies that $B(B')^{-1} \in \mathrm{K}$ if $B'$ also conjugates $M$ to a rotation.  Since $ s_+(x,\lambda + \iop  0)$ is that fixed point, the reader may note that 
\begin{equation} \label{eq:dirac:BEexplicitChoice}
B_\lambda
= \frac{1}{\sqrt{1-|s_+(x,\lambda + \iop  0)|^2}} \begin{bmatrix} 1 & -s_+(x,\lambda + \iop  0) \\[2mm] - \overline{s_+(x,\lambda + \iop  0)} & 1
\end{bmatrix}
\end{equation}
belongs to $\SU(1,1)$ and satisfies \eqref{eq:dirac:monodromyConjDef}.

\smallskip

(f) Consider $\Lambda_0$, the free Dirac operator (i.e., the Dirac operator with $\varphi\equiv 0$), and view it as a $T$-periodic operator. The associated monodromy is
\begin{equation}
M_z(x,0) = \begin{bmatrix} \eop^{-\iop  z T} & 0 \\ 0 & \eop^{\iop  z T} \end{bmatrix}.
\end{equation}
Thus, the associated discriminant is just $D(z) = 2\cos(Tz)$ and hence the bands are given by
\begin{equation} \label{eq:dirac:FREEBANDS} B_n = \left[ \frac{n \pi}{T}, \frac{(n+1)\pi}{T}  \right], \quad n \in \bbZ. \end{equation}
Choose $n \in \bbZ_+$ maximal with $n\pi/T \leq R + \|\varphi\|_\infty$. Combining \eqref{eq:dirac:FREEBANDS} with standard eigenvalue perturbation theory, at most $2(n+1)$ bands of $\Lambda_\varphi$ may then intersect $[-R,R]$. Since $n\pi/T \leq R+ \|\varphi\|_\infty$, the number of bands intersecting $[-R,R]$ is bounded from above by
\[
2(n+1) \leq 2\left(\frac{T}{\pi}(R + \|\varphi\|_\infty)+1\right),
\]
as desired.
\end{proof}

\subsection{A Helpful Formula for the Johnson-Moser Function}

Using the Dirac equation, we can express the logarithmic derivative of the dominant component of the Weyl solution,
\begin{align}\label{logderivativepsiplus}
\frac{\partial_x \psi^+_2(x,z)}{ \psi^+_2(x,z) }
& = \iop z - \iop  \overline{\varphi(x)} s_+(x,z), \\
- \frac{\partial_x \psi^-_1(x,z)}{ \psi^-_1(x,z) } 
& = \iop z - \iop  \varphi(x) s_-(x,z).
\end{align}
This motivates the definition of the Johnson--Moser function of a $T$-periodic Dirac operator,
\begin{equation} \label{eq:floquet:JMfctDef}
w(z) = \frac 1T \int_0^T  (  \iop z - \iop  \overline{\varphi(x)} s_+(x,z) ) \, {\rmd}x
\end{equation}
for $z \in \bbC_+$. Let us briefly note how $w(z)$ is related to objects already introduced.
\begin{prop} Suppose $\varphi \in C(\bbR)$ is $T$-periodic. For all $z \in \bbC_+$,
\begin{align} \label{eq:floquet:DztoWz}
D(z) & = 2 \cosh( Tw(z)) \\
\label{eq:floquet:LztoWz}
L(z) & = -\Re w(z).
\end{align}
\end{prop}

\begin{proof}
Note that $s_+(T,z) = s_+(0,z)$ by periodicity, so the vector $(s_+(0,z), 1)^\top$ is an eigenvector of the monodromy matrix with eigenvalue $\psi_2^+(T,z) / \psi_2^+(0,z)$. By integrating \eqref{logderivativepsiplus} it follows that the corresponding eigenvalue is equal to $\eop^{T w(z)}$. Due to $\det M_z(0,\varphi) = 1$, the other eigenvalue is $\eop^{-Tw(z)}$ and thus 
$$D(z) = \eop^{Tw(z)} + \eop^{-Tw(z)} = 2\cosh(Tw(z)),$$
proving \eqref{eq:floquet:DztoWz}.
\medskip

The previous argument shows $\eop^{Tw(z)} = \psi_2^+(T,z)/\psi_2^+(0,z)$ is one of the eigenvalues of the monodromy matrix. Since $\psi^+$ is the Weyl solution at $+\infty$, it follows that that $\eop^{-Tw(z)}$ is the dominant eigenvalue of $M_z(0,\varphi)$, so \eqref{eq:dirac:L(z)Def} implies
\[
- \Re w(z) = \frac 1T \log \left\lvert \frac{ \psi^+_2(0,z) }{ \psi^+_2(T,z) } \right\rvert = L(z),
\]
as promised.
 \end{proof} 
 
  Since the Lyapunov exponent defines a non-negative, symmetric, subharmonic function on $\bbC$ that is harmonic and positive on $\bbC\setminus\bbR$, it follows by complexifying \cite[Lemma 3.2]{EGL} that $w(z)$ is related to the density of states measure $\rho$ by
\[
-w(z) 
=c_0z + c_1+\int_{-1}^1  \log (\lambda-z) \, {\rmd}\rho(\lambda)+\int_{\bbR\setminus(-1,1)} \left( \log \left(1-\frac{z}{\lambda}\right)+\frac{z}{\lambda}\right) \, {\rmd}\rho(\lambda),
\]
where $c_0,c_1$ are some constants. Differentiating and passing from a normalization at $0$ to a normalization at $\iop $, we find $c_2$ so that 
\begin{equation}\label{derivativeofJMfunctionfromrho}
w'(z)=c_2+\int_{\bbR}\left(\frac{1}{\lambda-z}-\frac{\lambda}{1+\lambda^2}\right) \, {\rmd}\rho(\lambda),\quad c_2= \Re w'(\iop ).
\end{equation}
On the other hand, we will express it in terms of the full-line Green function. As a preliminary, we need:

\begin{lemma}
Let $z, z_0 \in \bbC \setminus \bbR$. 
\begin{enumerate}[label={\rm(\alph*)}]
\item Whenever $z \neq z_0$,
\begin{equation}\label{WeylWronskian}
\psi^+(x,z_0)^\top (- \iop J) \psi^+(x,z)  
= -  \iop  (z - z_0) \int_{x}^{+\infty} \psi^+(y,z_0)^\top \cJ \psi^+(y,z) \, {\rmd}y.
\end{equation}
\item If Weyl solutions are normalized by $\psi_2^+(x,z) = 1$, then
\begin{equation}\label{WeylWronskian2}
1 - \overline{s_+(x,z_0)} s_+(x,z)  = - \iop  (z - \overline z_0) \int_{x}^{+\infty} \psi^+(y,z_0)^* \psi^+(y,z) \, {\rmd}y.
\end{equation}
\item Weyl solutions are $L^2([x,\infty),\bbC^2)$-continuous on $\bbC_+$, that is, if we assume that Weyl solutions are normalized by $\psi_2^+(x,z) = 1$, then
\begin{equation}
\lim_{z \to z_0} \int_x^{+\infty} \lVert \psi^+(y,z) - \psi^+(y,z_0) \rVert^2 \, {\rmd}y = 0.
\end{equation}
\item For all $z\in \bbC_+$,
\begin{equation}\label{Weylderivative}
 \partial_z s_+(x,z)  =  \iop   \frac 1{ ( \psi^+_2(x,z) )^2} \int_{x}^{+\infty} \psi^+(y, z)^\top \cJ \psi^+(y,z) \, {\rmd}y.
\end{equation}
\end{enumerate}
\end{lemma}

\begin{proof}
(a) Recall that $A_z(x,\varphi)$ denotes the transfer matrix from $0$ to $x$; since $\varphi$ is fixed in this argument, we suppress it from the notation. Starting with the formula \eqref{eq:dirac:A*jAderiv}, we get
\[
A_{z_0}(x_1)^* j A_z(x_1) - A_{z_0}(x_2)^* j A_z(x_2) 
= - \iop  (z - \overline z_0) \int_{x_1}^{x_2} A_{z_0}(y)^* A_z(y) \, {\rmd}y.
\]
Multiplying by $\psi^+(0,z_0)^*$ on the left and by $\psi^+(0,z)$ on the right gives
\[
\psi^+(x_1,z_0)^* j \psi^+(x_1,z) - \psi^+(x_2,z_0)^* j \psi^+(x_2,z) 
= - \iop  (z - \overline z_0) \int_{x_1}^{x_2} \psi^+(y,z_0)^* \psi^+(y,z) \, {\rmd}y.
\]
To obtain an expression analytic in both $z$ and $z_0$, we note that by a direct verification \cite[Lemma 2.4]{EGL},  $\cJ \overline{ \psi^+(x,z_0) }$ is an eigensolution at energy $\overline{z_0}$ (with complex conjugation applied componentwise). It is also square-integrable on $[0,\infty)$ so it is a Weyl solution at energy  $\overline{z_0}$. Inserting
\begin{equation} \label{eq:diracgen:conjSolution}
\psi^+(x,\overline{z_0}) = \cJ \overline{ \psi^+(x,z_0) }
\end{equation}
we obtain
\[
\psi^+(x_1,\overline z_0)^\top \cJ  j \psi^+(x_1,z)  - \psi^+(x_2,\overline z_0)^\top \cJ  j \psi^+(x_2,z)  
= - \iop  (z - \overline z_0) \int_{x_1}^{x_2} \psi^+(y,\overline{z_0})^\top \cJ \psi^+(y,z) \, {\rmd}y.
\]
Replacing $\overline{z_0}$ by $z_0$ and using $\cJ j = -\iop  J$, we finally get a bianalytic expression
\[
\psi^+(x_1, z_0)^\top (-\iop J) \psi^+(x_1,z) - \psi^+(x_2, z_0)^\top (-\iop J) \psi^+(x_2,z)  
= - \iop  (z -  z_0) \int_{x_1}^{x_2} \psi^+(y,{z_0})^\top \cJ \psi^+(y,z) \, {\rmd}y.
\]
Due to the limit-point condition, the Wronskian $ \psi^+(x_2, z_0)^\top J \psi^+(x_2,z)$ decays as $x_2 \to \infty$, so taking $x_2 \to \infty$ and writing $x_1 = x$ gives \eqref{WeylWronskian}.
\medskip

(b)  By using the reflection symmetry \eqref{eq:diracgen:conjSolution} again on (a) we obtain
\[
\psi^+(x, z_0)^* j  \psi^+(x,z)    
= - \iop  (z -  \overline z_0) \int_{x}^{\infty} \psi^+(y,{z_0})^*  \psi^+(y,z) \, {\rmd}y.
\]
so in particular if $\psi^+_2(x,z) = \psi^+_2(x,z_0) = 1$ we obtain \eqref{WeylWronskian2}.
\medskip

(c) For any $z_0 \neq z$ we can expand the integral
\[
\int_x^{+\infty} \lVert \psi^+(y,z) - \psi^+(y,z_0) \rVert^2 \, {\rmd}y
\]
into four terms and apply part (b) to each term to obtain
\begin{align*}
\int_x^{+\infty} \lVert \psi^+(y,z) - \psi^+(y,z_0) \rVert^2 \, {\rmd}y  
& =  \frac{ 1 - \overline{s_+(x,z)} s_+(x,z) }{ - \iop (z - \overline z) }
- \frac{ 1 - \overline{s_+(x,z_0)} s_+(x,z) }{ - \iop (z - \overline z_0) } \\
& \quad  -  \frac{ 1 - \overline{s_+(x,z)} s_+(x,z_0) }{ - \iop (z_0 - \overline z) }
+ \frac{ 1 - \overline{s_+(x,z_0)} s_+(x,z_0) }{ - \iop (z_0 - \overline z_0) }
\end{align*}
and taking $z_0 \to z$, the right-hand side converges to $0$.
\medskip

(d) 
Dividing (a) by $\psi^+_2(x,z_0) \psi^+_2(x,z)$, we obtain
\[
s_+(x,z_0) - s_+(x,z) = -  \frac { \iop  (z -  z_0) }{\psi^+_2(x,z_0) \psi^+_2(x,z) } \int_{x}^{+\infty} \psi^+(y,{z_0})^\top \cJ \psi^+(y,z) \, {\rmd}y
\]
Dividing by $z_0 - z$ and taking the limit $z_0 \to z$ gives \eqref{Weylderivative}.
\end{proof}

We can now represent $\partial_z w(z)$ as an average of the trace of the matrix-valued diagonal Green functions: 

\begin{lemma} \label{lem:dzwDiagGreen} For a $T$-periodic Dirac operator, the derivative of the Johnson-Moser function is 
\begin{equation}\label{derivativeofJMfunction}
w'(z) =  \frac 1T \int_0^T  \frac{ \psi^-(x,z)^\top \cJ \psi^+(x,z) }{ W[\psi^+, \psi^-]}\, {\rmd}x, \qquad z\in \bbC_+.
\end{equation}
\end{lemma}

\begin{proof}
By a Cauchy formula and Fubini's theorem, it is allowed to differentiate the definition of the Johnson--Moser function inside the integral to obtain
\[
w'(z) 
= \frac 1T \int_0^T  (  \iop   - \iop  \overline{\varphi(x)} \partial_z s_+(x,z) ) \, {\rmd}x.
\]
To rewrite this, we essentially use an integration by parts. We start from an inspired guess
\[
h (x,z) 
= \frac{ s_-(x,z) \partial_z s_+(x,z)   }{ 1 - s_+(x,z) s_-(x,z) } 
= \frac{  \iop   \frac{ \psi_2^-(x,z) }{ \psi_1^-(x,z) }  \frac 1{ ( \psi^+_2(x,z) )^2} \int_{x}^{+\infty} \psi^+(y, z)^\top  \cJ  \psi^+(y,z) \, {\rmd}y }{\iop  \frac 1{ \psi_1^-(x,z) \psi_2^+(x,z) } \psi^+(x,z)^\top J \psi^-(x,z)  } .
\]
Writing $W =  \psi^+(x,z)^\top J \psi^-(x,z)$ for the Wronskian of $\psi^+$ and $\psi^-$, this simplifies to
\[
h (x,z)  = \frac { \psi_2^-(x,z)}{ W \psi_2^+(x,z) }   \int_{x}^{+\infty} \psi^+(y, z)^\top  \cJ  \psi^+(y,z)  \, {\rmd}y. 
\]
To differentiate this in $x$, we will use
\[
\partial_x \left( \frac { \psi_2^-(x,z)}{ \psi_2^+(x,z) }  \right) = \frac { \partial_x \psi_2^-(x,z) \psi_2^+(x,z) -  \psi_2^-(x,z) \partial_x \psi_2^+(x,z)  }{ ( \psi_2^+(x,z) )^2 }.
\]
Using the eigenfunction equation and simplifying gives
\[
\partial_x \left( \frac { \psi_2^-(x,z)}{ \psi_2^+(x,z) }  \right)  = \frac{ \overline{ \varphi(x) }   \psi^+(x,z)^\top J \psi^-(x,z) }{ ( \psi_2^+(x,z) )^2 }  = \frac{ \overline{ \varphi(x) } W}{ ( \psi_2^+(x,z) )^2 }
\]
so the derivative of $h$ is
\begin{align*}
\partial_x h(x,z) 
=  \frac{ \overline{ \varphi(x) } }{ ( \psi_2^+(x,z) )^2 } \int_{x}^{+\infty}  \psi^+(y, z)^\top  \cJ  \psi^+(y,z)   \, {\rmd}y -  \frac { \psi_2^-(x,z)}{ W \psi_2^+(x,z) }  \psi^+(x, z)^\top  \cJ  \psi^+(x,z)
\end{align*}
and computing  $\psi^+(x, z)^\top  \cJ  \psi^+(x,z) = 2 \psi_1^+(x,z) \psi_2^+(x,z)$, this simplifies to
\begin{align}
\partial_x h(x,z) = - \iop \overline{\varphi(x)} \partial_z s_+(x,z)  - 2 \frac{ \psi_1^+(x,z) \psi_2^-(x,z) } W.
\end{align}
The function $h(x,z)$ is $T$-periodic because $s_\pm(x,z)$ are $T$-periodic, so its derivative has zero integral on $[0,T]$. Thus, integrating $\partial_x h(x,z)$ from $0$ to $T$ gives
\[
\iop \int_0^T \overline{\varphi(x)} \partial_z s_+(x,z) \, {\rmd}x 
= - 2 \int_0^T \frac{ \psi_1^+(x,z) \psi_2^-(x,z) } W \, {\rmd}x
\]
so
\[
\partial_z w(z) 
= \iop   + \frac{ 2 }T  \int_0^T \frac{ \psi_1^+(x,z) \psi_2^-(x,z) } W \, {\rmd}x.
\]
Thus,
\[
\partial_z w(z) 
= \iop   + \frac{ 1}T  \int_0^T \frac{ 2 \psi_1^+(x,z) \psi_2^-(x,z) }{ \iop  (\psi_1^+(x,z) \psi_2^-(x,z) - \psi_2^+ (x,z) \psi_1^-(x,z) )   } \, {\rmd}x.
\]
Bringing $\iop $ into the integral and combining terms gives \eqref{derivativeofJMfunction}.
\end{proof}

We now combine those two formulas leads to get a helpful formula for the density of states.

\begin{lemma} \label{lem:dirac:DOSderiv} For a $T$-periodic Dirac operator, the density of states measure is absolutely continuous with respect to Lebesgue measure. Inside spectral bands {\rm(}i.e., for $\lambda$ such that $D(\lambda) \in (-2,2)${\rm)}, its Radon--Nikodym derivative is
\begin{equation}
\frac{{\rmd}\rho}{{\rmd}\lambda} (\lambda) = \frac 1{\pi T}   \int_0^T \frac {1} { 1 -   \lvert s_+(x,\lambda + \iop  0) \rvert^2 } \, {\rmd}x.
\end{equation}
\end{lemma}

\begin{proof}
In the formula \eqref{derivativeofJMfunction}, cancelling $\psi_1^- \psi_2^+$ from the fraction gives
\[
w'(z) = \frac{\iop }{T} \int_0^T  \frac{ 1+ s_-(x,z) s_+(x,z)}{ 1- s_-(x,z) s_+(x,z)} \, {\rmd}x
\]
and therefore
\[
\Im w'(z) 
= \frac 1T \int_0^T  \frac{ 1 - \lvert s_-(x,z) s_+(x,z) \rvert^2}{ \lvert 1- s_-(x,z) s_+(x,z) \rvert^2} \, {\rmd}x.
\]
On any spectral band, $s_\pm(x,z)$ have continuous boundary values in $\bbD$.  From \eqref{derivativeofJMfunctionfromrho}, by Stieltjes inversion, we can recover $\rho$ from $w'$. For $(a,b) \subseteq  \{ \lambda \in \bbR : D(\lambda,\varphi) \in (-2,2) \}$,
\[
\frac{ \rho((a,b)) + \rho([a,b])}2 =  \frac 1{\pi T} \int_a^b \int_0^T \frac {1 - \lvert  s_-(x,\lambda + \iop  0) s_+(x,\lambda + \iop  0) \rvert^2 } {\lvert 1 -   s_-(x,\lambda + \iop  0) s_+(x,\lambda + \iop  0) \rvert^2 } \, {\rmd}x \, {\rmd}\lambda
\]
Periodic operators are reflectionless, so for $D(\lambda) \in (-2,2)$, $s_-(x,\lambda + \iop 0)  = \overline{  s_+(x,\lambda + \iop  0) }$, which simplifies the formula to the final result.
\end{proof}

\begin{lemma} \label{l:ids:tm}
There is a constant $c_0>0$ with the following property. Suppose $\varphi \in C(\bbR)$ is $T$-periodic with $T >0$. Denote the associated discriminant by $D$ and the density of states by $\rho$. Inside spectral bands {\rm(}i.e., for $\lambda$ such that $D(\lambda) \in (-2,2)${\rm)}, one has
\begin{equation} \label{eq:ids:tm}
\frac{{\rmd}\rho}{{\rmd}\lambda}(\lambda)
\geq
\frac{c_0}{T} \int_0^T \! \|B_{\lambda}(x)\|^2 \, {\rmd}x,
\end{equation}
where $B_\lambda(x)$ is the conjugacy defined by \eqref{eq:dirac:monodromyConjDef}.
\end{lemma}

\begin{proof} Since $B_\lambda(x)$ is unique modulo left-multiplication by an element of $\mathrm{K}$, its Hilbert--Schmidt norm is independent of the choice of conjugacy. Since \eqref{eq:dirac:BEexplicitChoice} furnishes an example of a matrix that conjugates $M_\lambda(x)$ to a rotation (recall $s_+(x) = s_+(x,\lambda + \iop 0)$), we may explicitly compute the Hilbert--Schmidt norm of $B_{\lambda}(x)$
\begin{equation} \label{eq:hsnorm}
\| B_{\lambda}(x) \|_2^2
=
2\left(\frac{1+|s_+(x,\lambda + \iop  0)|^2}{1-|s_+(x,\lambda + \iop  0)|^2}\right).
\end{equation}
Since $|s_+| <1$, we get 
\begin{equation}
\| B_\lambda(x) \|_2^2 \leq \frac{4}{1-|s_+(x,\lambda + \iop  0)|^2},
\end{equation}
which gives the statement of the lemma upon applying Lemma~\ref{lem:dirac:DOSderiv} and recalling that all matrix norms on $2\times 2$ matrices are mutually equivalent.
\end{proof}

\subsection{Gordon Lemmas for Dirac Operators} \label{subsec:gordon}

In the present subsection, we explain how to exclude eigenvalues for Dirac operators with suitable repetitions in the potential $\varphi$. The central idea dates back to Gordon \cite{Gordon19761976} and has been implemented in a wide variety of settings over the years; see \cite{Damanik2000GordonSurvey, HanJitomirskaya2017Advances, Seifert2011EJDE, seifert2014preprint} for an incomplete list of examples.

The following lemma is well known. The proof is short, so we give it for the reader's convenience.

\begin{lemma} \label{lem:gordsl2c}
For any $M \in \SL(2,\bbC)$ and any $v \in \bbC^2$, one has
\begin{align}
\label{eq:gordsl2c1}
\max(\|M^{-1}v\|,\|Mv\|,\|M^2v\|) \geq \frac{1}{2}\|v\| \\
\label{eq:gordsl2c2}
\max(\|Mv\|,\|M^2v\|) \geq \frac{1}{2}\min\left(1, \frac{1}{|\tr\, M|}\right)\|v\|
\end{align}
\end{lemma}

\begin{proof}
By the Cayley--Hamilton theorem, one has
\begin{align}
\label{eq:gord:ch1}
M^2 &= (\tr \,M)M-I \\
\label{eq:gord:ch2}
M^{-1} &= (\tr\, M)I-M.
\end{align}
Applying \eqref{eq:gord:ch1} when $|\tr \,M|\leq 1$ and \eqref{eq:gord:ch2} when $|\tr \, M| >1$ proves \eqref{eq:gordsl2c1}, while \eqref{eq:gordsl2c2} follows directly from \eqref{eq:gord:ch1}.
\end{proof}

\begin{theorem}[3-Block Gordon Lemma] \label{t:gord:threeblock}
Let $\varphi \in L^\infty(\bbR)$ be given. If there exist $0<q_k \to \infty$ such that
\[ \varphi(x\pm q_k) = \varphi(x) \quad \forall 0 \le x < q_k, \]
then $\Lambda_\varphi$ has empty point spectrum.
\end{theorem}

\begin{proof}
The given assumption implies $A_z(2q_k,q_k,\varphi) =  A_z(q_k,0,\varphi) = A_z(0,-q_k,\varphi)$ for any $z$. Lemma~\ref{lem:gordsl2c} implies that any eigensolution $U$ at $z$ must satisfy
\begin{equation}
\|U(x_k)\| \geq \frac{1}{2}\|U(0)\|
\end{equation}
for a sequence $x_k \in \bbR$ with $|x_k| \to \infty$. By Gronwall's inequality
\[
\lVert U(x) \rVert \ge \lVert U(x_k) \rVert \exp\left( - \int_{x_k}^x ( \lVert \Phi(t) \rVert + \lvert z \rvert ) \, {\rmd}t \right)
\]
for $x > x_k$, with a similar estimate for $x < x_k$. In particular, since $\varphi\in L^\infty(\bbR)$,  there exists $\delta>0$ independent of $k$ such that $\|U(x)\| \geq \|U(0)\|/3$ for $x$ in the $\delta$-neighborhood of each $x_k$, which implies $U \notin L^2$.
\end{proof}

An additional assumption on the traces allows one to prove a result that requires fewer repetitions.

\begin{theorem}[2-Block Gordon Lemma]
Let $\varphi \in L^\infty(\bbR)$ be given. If there exist $0<q_k \to \infty$ such that
\[ \varphi(x + q_k) = \varphi(x) \quad \forall 0 \le x < q_k, \]
and if 
\begin{equation}\sup_k|\tr\, A_z(q_k,0,\varphi)|<\infty\end{equation}
then $z$ is not an eigenvalue of $\Lambda_\varphi$.
\end{theorem}

\begin{proof}
This follows from the same proof as Theorem~\ref{t:gord:threeblock}, but using the inequality \eqref{eq:gordsl2c2} instead of \eqref{eq:gordsl2c1}.
\end{proof}

Clearly, one can perturb around this and preserve the absence of eigenvalues.

\begin{definition} \label{def:gordonTypeDirac}
We say that $\varphi \in L^\infty(\bbR)$ is of \emph{Gordon type} if there exist $0<q_k\to\infty$ such that
\begin{equation} 
\label{eq:gord:gordfunctiondef}
\lim_{k\to\infty} C^{q_k} \sup_{0\le x < q_k}|\varphi(x - q_k) - \varphi(x) |
= \lim_{k\to\infty} C^{q_k} \sup_{0\le x < q_k}|\varphi(x+ q_k) - \varphi(x) |
= 0
\end{equation}
for every $C>0$.
\end{definition}

\begin{theorem} \label{t:gordonTypeDirac}
If $\varphi \in L^\infty(\bbR)$ is of Gordon type, then $\Lambda_\varphi$ has empty point spectrum.
\end{theorem}

\begin{proof}
Fix $k$ and let $\widetilde{\varphi}_k$ denote the $q_k$-periodic function that agrees with $\varphi$ on $[0,q_k)$.

Consider $z$ and a nontrivial solution $U$ of $\Lambda_\varphi U = zU$. By Lemma~\ref{lem:gordsl2c}, we can choose $x_k \in \{-q_k,q_k,2q_k\}$ such that
\begin{equation} \label{eq:gord:maintEst1}
\|A_z(x_k,0,\widetilde{\varphi}_k)U(0)\| \geq \frac{1}{2}\|U(0)\|.
\end{equation}
By variation of parameters,
\[
A_z(x_k,0,\varphi)-A_z(x_k,0,\widetilde{\varphi}_k) =  \int_0^{x_k} A_z(x_k,x,\widetilde{\varphi}_k) i j [ \tilde\Phi(x) - \Phi(x) ] A_z(x,0,{\varphi}) \, {\rmd}x
\]
and by Gronwall's inequality, $A_z(x_k, x, \widetilde{\varphi}_k) \le C^{x_k - x}$ and $A_z(x,0, {\varphi}) \le C^{x}$ for some $C$ depending on $z$ and $\lVert \varphi \rVert_\infty$, $\lVert \widetilde{\varphi}_k \rVert_\infty$. 
Thus, by \eqref{eq:gord:gordfunctiondef},
\begin{align} 
\nonumber
\| A_z(x_k,0,\varphi)-A_z(x_k,0,\widetilde{\varphi}_k)\| 
& \lesssim C^{q_k}\sup_{y \in I_k} |\varphi(y)  - \widetilde{\varphi}_k(y)| \\
\label{eq:gord:maintEst2}
& \to 0,
\end{align}
where $C>0$ is a suitable constant and $I_k = [0,x_k]$ or $[x_k,0]$ according to whether $x_k$ is positive or negative.
Putting together \eqref{eq:gord:maintEst1} and \eqref{eq:gord:maintEst2}, we get the following for large $k$:
\begin{align*}
\|U(x_k)\| & = \|A_z(x_k,0,\varphi)U(0)\| \\
& = \|A_z(x_k,0,\widetilde{\varphi}_k)U(0) + (A_z(x_k,0,\varphi)-A_z(x_k,0,\widetilde{\varphi}_k))U(0) \| \\
& \geq \frac{1}{2}\|U(0)\| - \| A_z(x_k,0,\varphi)-A_z(x_k,0,\widetilde{\varphi}_k)\|\|U(0) \| \\
& \geq \frac{1}{4}\|U(0)\|.
\end{align*}
The conclusion follows from Gronwall's inequality as before.
\end{proof}

\begin{prop} \label{prop:dirac:gordonGdelt}
Recall that $\LP(\bbR,\bbC)$ denotes the set of continuous uniformly limit-periodic functions $\bbR\to\bbC$. If $\mathcal G$ denotes the set of Gordon-type elements of $\LP(\bbR,\bbC)$, then $\mathcal G$ is a dense $G_\delta$ subset of $\LP(\bbR,\bbC)$.
\end{prop}

\begin{proof}
 Let ${\rm P}(\bbR,T)$ denote the set of $T$-periodic elements of $\LP(\bbR,\bbC)$, and observe that
\begin{equation}
\mathcal G 
= \bigcap_{N=1}^\infty \bigcup_{T \geq N} \bigcup_{\varphi \in P(\bbR,T)} B(\varphi,N^{-T}),
\end{equation}
where $B(\varphi,r)$ denotes the open ball of radius $r$ centered at $\varphi$. Since $\mathcal G$ is clearly dense, this shows that $\mathcal G$ is a dense $G_\delta$.
\end{proof}

\section{Proofs of Main Results for Dirac Operators} \label{sec:proofs}

\subsection{Proofs of Main Results}

We now put the results from the previous section to work to prove Theorem~\ref{t:dirac:zeromeas} and \ref{t:dirac:zerohd}. Given that limit-periodic operators are limits of periodic operators in the uniform topology, the key to producing the desired limit-periodic operators with thin spectra is a construction of \emph{periodic} operators with thin spectra. Here is the crucial lemma.

\begin{lemma} \label{lem:dirac:thinSpec}
Suppose $\varphi \in C(\bbR)$ is $T$-periodic. For all $R,\varepsilon > 0$, there exist $N_0 = N_0(\varphi,R,\varepsilon) \in \bbN$ and $c_0 = c_0(\varphi,R,\varepsilon) \in \bbN$ such that for every $N \in \bbN$ for which $N\geq N_0$, there in turn exists $\widetilde{\varphi} \in C(\bbR)$ of period $\widetilde{T} = NT$ such that
\begin{equation}
\|\varphi-\widetilde\varphi\|_\infty < \varepsilon
\end{equation}
and such that 
\begin{equation}
\Leb(\sigma(\Lambda_{\widetilde{\varphi}}) \cap [-R,R] )
\leq \eop^{-c_0\widetilde{T}}.
\end{equation}
\end{lemma}

Before proving Lemma~\ref{lem:dirac:thinSpec}, let us see how it implies the main statements. Let us note that these arguments deducing Theorems~\ref{t:dirac:zeromeas} and \ref{t:dirac:zerohd} from Lemma~\ref{lem:dirac:thinSpec} are standard and supplied for the reader's convenience. The key remaining challenge to overcome lies in the proof of Lemma~\ref{lem:dirac:thinSpec}.

\begin{proof}[Proof of Theorem~\ref{t:dirac:zeromeas}]
For $\varepsilon,R>0$, let $\mathcal M(\varepsilon,R) \subseteq \LP(\bbR,\bbC)$ denote the set of $\varphi$ such that $\Leb(\sigma(\Lambda_\varphi)\cap [-R,R]) < \varepsilon$. Since $\mathcal M(\varepsilon,R)$ is dense by Lemma~\ref{lem:dirac:thinSpec} and open by Proposition~\ref{prop:dirac:HdSpecPert}, it follows that 
\begin{equation}
\mathcal M = \bigcap_{n \in \bbN}\mathcal  M(1/n,n)
\end{equation}
is a dense $G_\delta$ subset of $\LP(\bbR,\bbC)$. For each $\varphi \in \mathcal  M$, the spectrum has zero Lebesgue measure and hence empty interior. By general principles, the spectrum of $\Lambda_\varphi$ does not have isolated points whenever $\varphi$ is limit-periodic \cite{Pastur1980CMP} (Pastur works with ergodic Schr\"odinger operators, but the proof applies in the case of ergodic Dirac operators with cosmetic modifications).

Thus, $\sigma(\Lambda_\varphi)$ is a zero-measure Cantor set for every $\varphi \in \mathcal M$.

Let $\mathcal G$ denote the set of $\varphi \in \LP(\bbR, \bbC)$ of Gordon type (cf.\ Definition~\ref{def:gordonTypeDirac}). We know that $\mathcal G$ is residual by Proposition~\ref{prop:dirac:gordonGdelt}, so we claim that the desired residual set is given by $\mathcal M \cap\mathcal  G$. Given $\varphi \in\mathcal  M \cap \mathcal  G$, the spectral type of $\Lambda_\varphi$ is purely singular since $\varphi \in \mathcal  M$ implies the spectrum has zero Lebesgue measure (and hence cannot support absolutely continuous measures). On the other hand, the spectral type is purely continuous by Theorem~\ref{t:gordonTypeDirac}. Since $\sigma(\Lambda_\varphi)$ is a zero-measure Cantor set for $\varphi \in \mathcal  M$, the proof is done.
\end{proof}

Next we discuss density of the set of $\varphi \in \LP(\bbR, \bbC)$ for which $\sigma(\Lambda_\varphi)$ has zero Hausdorff dimension. Let us briefly recall how Hausdorff measures and dimension are defined; for further details, see Falconer \cite{Falconer}.

Given a set $S \subseteq \bbR$ and an $\alpha \geq 0$, one defines the $\alpha$-dimensional \emph{Hausdorff measure} of $S$ by
\begin{equation}
h^\alpha(S)
=
\lim_{\delta \downarrow 0} \inf \set{\sum_j |I_j|^\alpha : \{I_j\} \text{ is a } \delta\text{-cover of S} },
\end{equation}
where $\{I_j\}$ is a $\delta$-cover of $S$ if and only if $|I_j|<\delta$ for all $j$ and $S$ is contained in the union of all $I_j$. For each $S \subseteq \bbR$, there is a unique $\alpha_0 \in [0,1]$ with the property that
\[
h^\alpha(S)
=
\begin{cases}
\infty & 0 \leq \alpha < \alpha_0 \\
0      & \alpha_0 < \alpha
\end{cases}
\]
We denote $\alpha_0 = \dim_{\rm H}(S)$ and refer to this value as the \emph{Hausdorff dimension} of the set $S$. 

Let us also briefly recall some definitions related to box-counting dimension. For a \emph{bounded} set $S \subseteq \bbR$ and $\varepsilon>0$, write $N(S,\varepsilon)$ for the minimal number of intervals of length $\varepsilon$ needed to cover $S$. The upper and lower box-counting dimensions of $S$ are given by
\begin{equation}
\dim_{\rm B}^+(S) = \limsup_{\varepsilon \downarrow 0} \frac{\log N(S,\varepsilon)}{\log(1/\varepsilon)}, \quad
\dim_{\rm B}^-(S) = \liminf_{\varepsilon \downarrow 0} \frac{\log N(S,\varepsilon)}{\log(1/\varepsilon)}.  
\end{equation}
When $S$ is bounded, one has
\begin{equation}
\dim_{\rm H}(S) \leq \dim_{\rm B}^-(S) \leq \dim_{\rm B}^+(S).
\end{equation}

\begin{proof}[Proof of Theorem~\ref{t:dirac:zerohd}]
Let $\varphi \in \LP(\bbR,\bbC)$  and $\varepsilon>0$ be given. Since the periodic elements are dense, we may assume without loss that $\varphi$ is $T$-periodic for some $T>0$. 

 Define $\varphi^{(0)} =\varphi$ and $\varepsilon_0 = \varepsilon/2$. Using Lemma~\ref{lem:dirac:thinSpec} inductively we may construct for $n \in \bbN$ operator data $\varphi^{(n)} \in \LP(\bbR,\bbC)$ of period $T_n$ and $\varepsilon_n>0$ such that
\begin{equation}
\|\varphi^{(n)} - \varphi^{(n+1)}\|_\infty < \varepsilon_n, \quad n \geq 0
\end{equation}
where
\begin{equation} \label{eq:dirac:hd0epschoice}
0 < \varepsilon_n < \min\left( \frac{\varepsilon_{n-1}}{2}, \frac{1}{2} (n+1)^{-T_n}, \frac{1}{4}\Leb([-n,n] \cap \sigma(\Lambda_{\varphi^{(n)}})) \right), \quad n \geq 1
\end{equation}
and such that
\begin{equation}
\Leb([-n,n] \cap \sigma(\Lambda_{\varphi^{(n)}}) \leq \exp(-T_n^{1/2}), \quad n \geq 1.
\end{equation}
Let us abbreviate $\Lambda_n = \Lambda_{\varphi^{(n)}}$ and $\Sigma_n = \sigma(\Lambda_n)$ for $0 \le n \le \infty$. By \eqref{eq:dirac:hd0epschoice} and the choice of $\varepsilon_0$, one has
\begin{align*}
\|\varphi - \varphi^{(\infty)}\|_\infty \leq \sum_{n=0}^\infty \varepsilon_n < \sum_{n=0}^\infty 2^{-n-1}\varepsilon = \varepsilon.
\end{align*}
From \eqref{eq:dirac:hd0epschoice}, we also see that
\[\|\varphi^{(m)} - \varphi^{(\infty)}\|_\infty
 \leq \sum_{n=m}^\infty \varepsilon_n < (m+1)^{-T_m}, \]
which implies that $\Lambda_\infty$ satisfies the Gordon criterion and hence has purely continuous spectrum.

As before, $\sigma(\Lambda_\infty)$ lacks isolated points and cannot support absolutely continuous measures if it has zero Lebesgue measure, so all that remains is to prove $\dim_{\rm H}(\Sigma_\infty) = 0$.

Fix $n \in \bbN$. Applying \eqref{eq:dirac:hd0epschoice} and Proposition~\ref{prop:dirac:HdSpecPert}, for each $k \geq n$, the $2\varepsilon_k$-neighborhood of $[-n,n] \cap \Sigma_k$ covers $[-n,n] \cap \Sigma_\infty$ and hence writing this covering as $\{I_j\}$, one has a covering of $[-n,n] \cap \Sigma_\infty$ by intervals of length at most $2\eop^{-T_k^{1/2}}$. By Theorem~\ref{t:dirac:floquet}, this covering consists of $\lesssim T_k$ intervals. This leads to
\begin{align*}
\dim_{\rm B}^-([-n,n] \cap \Sigma_\infty) 
= \liminf_{\varepsilon\downarrow 0} \frac{\log N([-n,n] \cap \Sigma_\infty,\varepsilon)}{\log\varepsilon^{-1}} 
& \leq \limsup_{k\to\infty} \frac{\log N(\Sigma_\infty,2 \eop^{- T_k^{1/2}})}{\log[2 \eop^{T_k^{1/2}}]} \\
& \leq \limsup_{k\to\infty} \frac{\log T_k}{\sqrt{T_k}}\\
& =0.
\end{align*}
 which suffices to show $\dim_{\rm B}^-([-n,n] \cap \Sigma_\infty)=0$ for each fixed $n \in \bbN$.
 
  From
\[ 0 \le \dim_{\rm H}([-n,n] \cap \Sigma_\infty) \leq \dim_{\rm B}^-([-n,n] \cap \Sigma_\infty)  = 0, \]
we see that $[-n,n] \cap \Sigma_\infty$ has zero Hausdorff dimension as well. Sending $n \to \infty$, $\Sigma_\infty$ has Hausdorff dimension zero, as desired.
\end{proof}

\subsection{Proof of Lemma~\ref{lem:dirac:thinSpec}}
Now let us work on Lemma~\ref{lem:dirac:thinSpec}. The first crucial ingredient is the following noncommutation lemma. Although it is straightforward to state and prove, we want to highlight this, as it is the key conceptual idea that made the analysis of the present work possible. Recall that an element $A \in \SU(1,1)$ is called \emph{elliptic} if $\tr\, A \in (-2,2)$ and \emph{hyperbolic} if $\tr \, A \in \bbR \setminus [-2,2]$. An elliptic matrix $A$ is characterized by having complex conjugate eigenvalues $\lambda_\pm = \eop^{\pm \iop  \theta}$ for some $\theta \in (0,\pi)$ while hyperbolic matrices are charaterized by having eigenvalues $\lambda_\pm = \lambda^{\pm 1}$, where  $\lambda \in \bbR \setminus [-1,1]$.

\begin{lemma} \label{lem:SU11:nonabSemiGroup}
Suppose $A,B \in \SU(1,1)$ are elliptic and $[A,B]\neq 0$. Then the semigroup generated by $A$ and $B$ contains a hyperbolic element of $\SU(1,1)$.
\end{lemma}

\begin{proof}
It is well-known that the \emph{closed subgroup} of $\SU(1,1)$ generated by $A$ and $B$ contains a hyperbolic element  \cite[Lemma~10.4.14]{Simon2005OPUC2}.

Now, let $A$ be elliptic with eigenvalues $\eop^{\pm 2\pi \iop  t}$, $t \in [0,1/2]$. If $t = p/q$ is rational, then one has $A^{-1}= A^{q-1}$. If $t$ is irrational, then choosing rationals $p_k/q_k \to t$ one can check that $A^{q_k-1} \to A^{-1}$, so one can approximate $A^{-1}$ to arbitrary precision with positive powers of $A$. In either case, the closed semigroup generated by two elliptic matrices is the same as the closed subgroup they generate.

Since the closed semigroup generated by $A, B$ contains a hyperbolic element and the set of hyperbolic matrices is open, the semigroup generated by $A, B$ contains a hyperbolic element.
\end{proof}

By conjugating with a Cayley transform, one can immediately push this result to $\SL(2,\bbR)$. Though we do not need it for the present work, we record it here, since it may be of independent interest.

\begin{coro}
Suppose $A,B \in \SL(2,\bbR)$ are elliptic and $[A,B]\neq 0$. Then the semigroup generated by $A$ and $B$ contains a hyperbolic element of $\SL(2,\bbR)$.
\end{coro}
\begin{proof}
Recall that $A \in \SL(2,\bbR) \iff WAW^* \in \SU(1,1)$, where
\begin{equation}
 W = -\frac{1}{1+\iop } \begin{bmatrix} 1 & -\iop  \\ 1 & \iop  \end{bmatrix}. \end{equation}
Thus, the result follows by applying Lemma~\ref{lem:SU11:nonabSemiGroup} to $WAW^*$ and $WBW^*$.
\end{proof}

In view of Lemma~\ref{lem:SU11:nonabSemiGroup}, it is helpful to know when matrices commute.
For a set $H \subseteq \SL(2,\bbC)$, denote by $Z(H) = \{g \in \SL(2,\bbC) : gh=hg \ \forall h \in H\}$ its centralizer. For $T>0$, let $Q(T,z)$ denote the set of all periodic Dirac transfer matrices over period $T$ with energy $z$, that is, 
\begin{equation}
Q(T,z) 
= \set{A_z(T,\varphi) : \varphi \in C(\bbR) \text{ is } T\text{-periodic}}.
\end{equation}

\begin{lemma} \label{lem:dirac:ZHT}
For each $T>0$ and $z \in \bbC$, 
$Z(Q(T,z)) = \{\pm I\}$.
\end{lemma} 

\begin{proof}
Let $z \in \bbC$ and $T>0$ be given, and consider $\varphi(x) \equiv c$ the constant function. The corresponding transfer matrices are given by solving
\begin{equation}
\begin{bmatrix}
\iop  & 0 \\ 0 & - \iop 
\end{bmatrix}\partial_x A + \begin{bmatrix}
0 & c \\ \overline{c} & 0
\end{bmatrix} A = zA
\end{equation}
leading to
\begin{equation}
\partial_x A 
= \underbrace{\begin{bmatrix}
-\iop z & \iop c \\ -\iop \overline{c} & \iop z
\end{bmatrix}}_{=:B(z,c)}A
\end{equation}
which gives the transfer matrix $A_z(T,c) = \eop^{T B(z,c)}$. The matrix $B(z,c)$ has eigenvectors
\begin{equation}
v_\pm = \begin{bmatrix}
z \pm \sqrt{z^2-|c|^2} \\ \overline{c}
\end{bmatrix}
\end{equation}
so any matrix $D\in Z(Q(T,z))$ must have the same eigenvectors. By varying $c$, we see that $D$ has more than two linearly independent eigenvectors that are not multiples of each other. Thus, $D = \pm I$.
\end{proof}

We will now combine this with analyticity. The transfer matrix $A_z(T,\varphi)$ is a solution of the initial value problem so it is an analytic function of $\varphi,\overline \varphi$ (see \cite{GrebertKappeler2014}): it can be represented as a convergent power series
\[
A_z(T,\varphi) = \sum_{n=0}^\infty  P_n(\varphi,\bar\varphi)
\]
where each $P_n$ is homogeneous of degree $n$  in $\varphi,\bar\varphi$. It is therefore also a real analytic function of $\Re \varphi, \Im \varphi \in C([0,T], \bbR)$ with values in the space of $2\times 2$ matrices, viewed as a real Banach space. 

\begin{lemma} \label{lem:dirac:resolventCover}
For every $T$-periodic $\varphi_0 \in C(\bbR)$, $\lambda \in \bbR$, and $\varepsilon > 0$, there exists $\widetilde{\varphi} \in C(\bbR)$ of period $\widetilde{T} \in T\,\bbN$ such that $\widetilde\varphi(0) = \varphi(0)$,
\begin{equation}
\|\varphi_0 -\widetilde\varphi\|_\infty < \varepsilon,
\end{equation}
and
\begin{equation}
\lambda \notin \sigma(\Lambda_{\widetilde\varphi}).
\end{equation}
\end{lemma}

\begin{proof}
We begin by noting that it suffices to prove existence of $\widetilde\varphi$ satisfying all conditions except $\widetilde\varphi(0) = \varphi(0)$; that condition can be added by a continuous perturbation in a neighborhood of $0$ which can furthermore be chosen to be an arbitrarily small perturbation in $L^2([0,\widetilde T])$ norm. Since gap edges of the $\widetilde T$-periodic operator are continuous with respect to $L^2([0,\widetilde T])$ norm, this can be done without affecting the other conclusions.

From here, the argument falls into three cases, according to the monodromy matrix $M_\lambda(\varphi_0) = A_\lambda(T,\varphi_0)$.
\medskip

\textbf{Case 1: \boldmath $\tr \, M_\lambda(\varphi_0) \in \bbR \setminus [-2,2]$.} In this case, $\lambda \notin \sigma(\Lambda_{\varphi_0})$, so set $\widetilde\varphi=\varphi_0$.

\medskip
\textbf{Case 2: \boldmath $\tr \, M_\lambda(\varphi_0) \in (-2,2)$}. In this case $M_\lambda(\varphi_0) \neq \pm I$ so, by Lemma~\ref{lem:dirac:ZHT}, there exists a $T$-periodic $\varphi_1 \in C(\bbR)$ such that
\begin{equation}\label{10dec1}
[M_\lambda(\varphi_0),M_\lambda(\varphi_1)] \neq 0.
\end{equation}
This commutator is a real analytic function of $\varphi_1 \in C([0,T])$, viewed as a real Banach space. Since the commutator is not identically zero, its zero set has empty interior by the identity principle for real analytic functions \cite[3.1.24]{Federer}. Moreover, elliptic elements form an open set; thus, there exists $\varphi_1 \in C([0,T])$ with $\|\varphi_0 -\varphi_1 \|_\infty < \varepsilon$ such that \eqref{10dec1} holds and both $M_\lambda(\varphi_0)$ and $M_\lambda(\varphi_1)$ are elliptic. As discussed at the beginning of the proof, we can wiggle $\varphi_1$ slightly to ensure $\varphi_0(0) = \varphi_1(0)$. 
By Lemma~\ref{lem:SU11:nonabSemiGroup} the semigroup generated by $M_\lambda(\varphi_0)$ and $M_\lambda(\varphi_1)$ contains a hyperbolic element, so we may choose $L \in \bbN$, $k_1,k_2,\ldots,k_L \in \bbN$, and $s_1,s_2,\ldots ,s_L\in\{0,1\}$ so that the matrix
\begin{equation}
\widetilde{A} = M_\lambda(\varphi_{s_L})^{k_L}M_\lambda(\varphi_{s_{L-1}})^{k_{L-1}} \cdots M_\lambda(\varphi_{s_1})^{k_1}
\end{equation}
is hyperbolic. The desired perturbation of $\varphi_0$ is the corresponding concatenation of $\varphi_0$ and $\varphi_1$. Define $\widetilde\varphi$ by
\begin{equation}
\widetilde{\varphi}(x) = \varphi_{s_\ell}(x) \quad \text{ whenever } \quad
T \sum_{j=1}^{\ell-1}k_j \leq x < T\sum_{j=1}^\ell k_j,
\end{equation}
and extend $\widetilde{\varphi}$ to a $KT$-periodic function, where $K = k_1+\cdots+k_L$.

\medskip

\textbf{Case 3: \boldmath $\tr \, M_\lambda(\varphi_0) = \pm 2$.}  The entries of $M_\lambda(\varphi)$ and its trace are  real analytic functions of $\varphi$. Since neither of the equalities $\tr \, M_\lambda(\varphi)= \pm 2$ hold identically in $\varphi$, the set where it holds has empty interior, again by the identity principle for real analytic functions. Thus, by an arbitrarily small perturbation of $\varphi_0$ we reduce to Case~1 or Case~2.
\end{proof}

Now, all that remains is to prove the key lemma.

\begin{proof}[Proof of Lemma~\ref{lem:dirac:thinSpec}]
Let $\varphi$, $R$, and $\varepsilon>0$ be given. By Lemma~\ref{lem:dirac:resolventCover} and compactness, we may find $\{\varphi_j\}_{j=1}^m$ with period $T' \in T\bbN$ such that 
\[ \bigcup_{j=1}^m \rho(\Lambda_{\varphi_j}) \supseteq [-R,R] \]
and
\[\varphi(0) = \varphi_j(0), \quad j = 1,2,\ldots,m\]
(of course, the periods may not be the same at first; however, since all periods are multiples of $T$, one can clearly pass to the least common multiple of the finite collection). 

Let $L(\lambda,\varphi)$ refer to the Lyapunov exponent at energy $\lambda$ associated with the operator $\Lambda_\varphi$. Since $L(\lambda,\varphi)$ is a continuous function of $\lambda \in \bbR$ that is positive away from the spectrum of $\Lambda_\varphi$, the construction gives
\begin{equation} \label{eq:thin:etaDef}
\kappa := \min_{|\lambda| \leq R} \max_{1 \leq j \leq m} L(\lambda,  \varphi_j) > 0.
\end{equation}
Consider $N \in \bbN$ large, and choose $\hat{N}$ maximal with $m(\hat{N} + 1)T' \leq NT$. Thus, 
\begin{equation} \label{eq:dirac:largeNchoice1}
 \hat{N}T'>\frac{NT}{m}- 2T' \geq \frac{NT}{2m},
\end{equation}
where the second equality holds if $N$ is sufficiently large.
 Define $\widetilde{\varphi}$ to be $\widetilde{T} = NT$-periodic by concatenating a total of $\hat{N}+1$ copies of each $\varphi_j$ and filling in the remainder with $\varphi$'s. More precisely, let $s_j = j(\hat{N}+1)T'$, put
\begin{equation}
\widetilde\varphi = \begin{cases} \varphi_j(x) & s_{j-1} \leq x < s_j \\
\varphi(x) & s_m \leq x < \widetilde{T}, \end{cases}
\end{equation}
and extend to a $\widetilde{T}$-periodic function. By construction, one can see that $\widetilde{\varphi}$ is continuous and satisfies $\|\varphi-\widetilde\varphi\|_\infty < \varepsilon$.

By a repetition of the arguments of \cite{Avila2009CMP, DFL2017JST}, one arrives at
\begin{equation} \label{eq:dirac:dflgoal}
 \Leb(\sigma(\Lambda_{\widetilde{\varphi}}) \cap [-R,R] ) 
 \leq \eop^{-c\widetilde{T}}\end{equation}
for a suitable constant $c$. 

For the reader's convenience, let us supply the details. Recall that $M_\lambda(s) = A_\lambda(s+\widetilde{T},s,\widetilde\varphi)$ denotes the mondromy matrix starting at $s \in \bbR$ and $D(\lambda) = \tr(M_\lambda(s))$ denotes the discriminant. Let $\lambda \in [-R,R]$ be given such that $D(\lambda) \in (-2,2)$, and use \eqref{eq:thin:etaDef} to choose $j$ with $1 \le j \le m$ such that $L(\lambda,\varphi_j) \geq \kappa$.

For $s \in \bbR$, define also
\begin{align*}
X_{\lambda}(s) & = A_\lambda(s + \hat{N}T',s,\widetilde\varphi). \end{align*}
Thus, $X_{\lambda}(s)$ transfers across a subinterval of length $\hat{N}T'$ beginning at $s$. By $\widetilde{T}$-periodicity of $\widetilde{\varphi}$, the reader can readily check that
\begin{align} 
\nonumber
& \, X_{\lambda}(s)^{-1} M_{\lambda}(s + \hat{N}T') X_{\lambda}(s)  \\
\nonumber
 = & \,  A_{\lambda}(s+\hat{N}T',s,\widetilde\varphi)^{-1}  A_{\lambda}(s+\hat{N}T' +\widetilde{T},s+\hat{N}T',\widetilde\varphi)  A_{\lambda}(s+\hat{N}T',s,\widetilde\varphi) \\ 
 \nonumber
 = & \,  A_{\lambda}(s,s+\hat{N}T',\widetilde\varphi)  A_{\lambda}(s+\hat{N}T'+\widetilde{T},s,\widetilde\varphi) \\ 
  \nonumber
 = & \,  A_{\lambda}(s+\widetilde{T},s+\hat{N}T' +\widetilde{T},\widetilde\varphi)  A_{\lambda}(s+\hat{N}T' +\widetilde{T},s,\widetilde\varphi) \\
   \nonumber
 = & \,  A_{\lambda}(s+\widetilde{T},s,\widetilde\varphi) \\
\label{eq:dirac:XlsMlsRel}
 = & \,  M_{\lambda}(s).
\end{align}
for any $s$.

For $s \in [s_j,s_j+T']$, notice that $\widetilde\varphi$ coincides with $\varphi_j$ on the interval $[s,s+\hat{N}T']$, so we have the following lower bound for every $s \in [s_j,s_j+T']$:
\begin{align}
\nonumber
\|X_{\lambda}(s)\|
& = \|A_\lambda(s+\hat{N}T',s,\widetilde\varphi)\| \\
\nonumber
& =  \|A_\lambda(s+T',s,\varphi_j)^{\hat{N}}\| \\
\nonumber
& \geq \mathrm{spr}(A_\lambda(s+T',s, \varphi_j))^{\hat{N}} \\
\nonumber
& = \eop^{L(\lambda,\varphi_j) \hat{N}T'} \\
\label{eq:dirac:XlambdasLB}
& \geq
\eop^{\kappa \hat{N}T'}.
\end{align}
Combining \eqref{eq:dirac:XlambdasLB} with \eqref{eq:dirac:largeNchoice1}
\begin{equation}
\|X_\lambda(s)\| 
\geq \eop^{c_1\widetilde{T}}, \quad \forall \ s \in [s_j,s_j+T']
\end{equation}
where $c_1 = \kappa/(2m) >0$.

Since $|D(\lambda)|<2$, let us denote by $B_{\lambda}(s)$ the conjugacies from Theorem~\ref{t:dirac:floquet} such that $B_{\lambda}(s) M_{\lambda}(s) B_{\lambda}(s)^{-1} \in \mathrm{K}$ (where $\mathrm{K}$ denotes the diagonal elements of $\SU(1,1)$ as in \eqref{eq:dirac:monodromyConjDef}). 
From \eqref{eq:dirac:XlsMlsRel}, we see that $M_{\lambda}(s)$ is conjugated to a rotation by both $B_{\lambda}(s)$ and $B_{\lambda}(s+\hat{N}T')X_{\lambda}(s)$.  By uniqueness of conjugacies modulo diagonal rotations, we have
\begin{equation}
B_{\lambda}(s+\hat{N}T')X_{\lambda}(s) = QB_{\lambda}(s)
\end{equation}
for some diagonal $Q \in \SU(1,1)$. This implies
\begin{equation}
\max(\|B_{\lambda}(s) \|, \|B_{\lambda}(s+\hat{N}T')\|) \geq \eop^{c_1 \widetilde{T}/2}, \quad
s \in [s_j,s_j+T'].
\end{equation}
According to Lemma~\ref{l:ids:tm}, this gives
\[
\frac{{\rmd}\rho}{{\rmd}\lambda}(\lambda) \geq c_0\widetilde{T}^{-1} \int_0^{\widetilde{T}} \|B_{\lambda} (s)\|^2 \, {\rmd}s
\geq c_0T'\widetilde{T}^{-1} \eop^{c_1 \widetilde{T}}.
\]
Since the density of states measure gives weight $1/\widetilde{T}$ to each band of the spectrum, each band in $[-R,R]$ has measure at most
\begin{equation}
|B| \lesssim (T')^{-1} \eop^{-c_1\widetilde{T}}.
\end{equation}
In view of Theorem~\ref{t:dirac:floquet}, there are $\lesssim \widetilde{T}$ such bands, and hence we obtain
\[ \Leb(\sigma(\Lambda_{\widetilde\varphi}) \cap [-R,R])
\lesssim (T')^{-1}\widetilde{T} \eop^{-c_1\widetilde{T}}.\]
Choosing a constant $0<c<c_1$ and $N$ sufficiently large, we obtain \eqref{eq:dirac:dflgoal} and hence the lemma is proved.
\end{proof}

\section{The CMV Setting} \label{sec:OPUC}
To demonstrate the versatility of this approach, we show how it can be used to resolve the corresponding issue for CMV matrices. Indeed, it is known that new ideas were necessary to resolve the question of zero-measure Cantor spectrum for limit-periodic CMV matrices. Compare the discussion on \cite[pp.~5113--5114]{FillmanOng2017JFA}.

The main new technique here is to apply the noncommutation idea from the Dirac operator setting to perturb spectra of CMV matrices and introduce spectral gaps in a controlled fashion. As before, the construction of periodic sequences of $\alpha$'s with thin spectra is the crucial technical step. The overall program is similar (but technically simpler in a few steps) for the CMV setting.

Recall that $\delta$ denotes the hyperbolic metric on $\bbD$ and the induced uniform metric on $\bbD^\bbZ$ as in Definition~\ref{def:poincaremetric}.

\begin{lemma} \label{lem:OPUC:thinSpec}
For any $q$-periodic $\alpha \in \bbD^\bbZ$ and $\varepsilon>0$, there exist $c_0=c_0(\alpha,\varepsilon) > 0$ and $N_0 = N_0(\alpha,
\varepsilon)\in \bbN$ such that for any $N\geq N_0$, there exists $\widetilde\alpha \in \LP(\bbZ,\bbD)$ of period $\widetilde{q}=Nq$ such that
\begin{equation}
\delta(\alpha,\widetilde\alpha) < \varepsilon
\end{equation}
and
\begin{equation}
\Leb(\sigma(\calE_{\widetilde\alpha})) \leq \eop^{-c_0\widetilde{q}}
\end{equation}
\end{lemma}

As in the Dirac case, Lemma~\ref{lem:OPUC:thinSpec} yields the desired results.

\begin{proof}[Proof of Theorems~\ref{t:extCMV:zeromeas} and \ref{t:extCMV:zerohd}]
Theorems~\ref{t:extCMV:zeromeas} and \ref{t:extCMV:zerohd} follow from Lemma~\ref{lem:OPUC:thinSpec} in precisely the same manner that Theorems~\ref{t:dirac:zeromeas} and \ref{t:dirac:zerohd} followed from Lemma~\ref{lem:dirac:thinSpec}. The relevant version of the Gordon lemma in the CMV case is given in \cite{Fillman2017PAMS}. In fact, the arguments are very slightly simpler, since $\partial \bbD$ is compact, so there is no need to work locally in energy in this setting.
\end{proof}

The remainder of the section is concerned with the proof of Lemma~\ref{lem:OPUC:thinSpec}, which is similar to that of Lemma~\ref{lem:dirac:thinSpec}; we concentrate on the key steps. Let us introduce some tools and characters.
Given $a \in \bbD$ and $z \in \bbC$, the Szeg\H{o} matrix is given by 
\begin{equation}
A(a,z) = \frac{1}{\sqrt{1-|a|^2}} \begin{bmatrix}
z & -\bar{a} \\ -az & 1
\end{bmatrix}.
\end{equation}
For $n,m \in \bbZ$, we also define
\[A_z(n,m,\alpha) 
= \begin{cases} A(\alpha_{n-1},z)\cdots A(\alpha_m,z) & n>m \\ 
I & n=m \\ 
[A_z(m,n,\alpha)]^{-1} & n<m . \end{cases}\]
Of course, the final line is not well-defined if $z=0$, but this is not an issue since we will only consider $z \in \partial \bbD$.

If $\alpha$ is $q$-periodic, the \emph{monodromy matrix} is given by $M_z = M_z(\alpha) = z^{-q/2}A_z(q,0,\alpha)$ and determines the spectrum via
\[ \sigma(\mathcal{E}_\alpha) = \set{z \in \partial \bbD : \tr(M_z) \in [-2,2]} \]
in that case. For this and other facts about periodic CMV matrices, we direct the reader to Simon \cite{Simon2005OPUC2}.

As we did with Dirac operators, we will use real-analyticity of the discriminant and Szeg\H{o} matrices as functions of $\Re \alpha_n$ and $\Im \alpha_n$.
The following identity principle supplies the needed input. For more details about multivariate analytic functions, we direct the reader to Gunning--Rossi \cite{GunningRossi1965}, particularly Theorem~6 of Chapter~1.

\begin{theorem}
If a real-analytic function of $n$ variables vanishes on an open subset of $\bbR^n$, then it vanishes identically.
\end{theorem}

As a consequence of this, we can deduce the following helpful fact.

\begin{prop} \label{prop:OPUC:noncommuting}
Suppose $\alpha \in \bbD^\bbZ$ is $q$-periodic, $z \in \partial\bbD$, $M_z(\alpha) \neq \pm I$. For every $\varepsilon>0$, there exists $\beta$ of period $q$ such that $\delta(\alpha, \beta) <\varepsilon$ and 
\begin{equation}
[M_z(\alpha), M_z(\beta)] \neq 0.
\end{equation}
\end{prop}

\begin{proof}
Suppose that for some $q$-periodic $\alpha \in \bbD^\bbZ$ and some $z \in \partial \bbD$, the function
\begin{equation} \label{eq:cmv:commMzalphabetaDef}
\bbD^q \ni \beta \mapsto [M_z(\alpha), M_z(\beta)] \in \bbC^{2\times 2}
\end{equation}
vanishes\footnote{In a minor abuse of notation, we write $\beta$ both for the element of $\bbD^q$ and the obvious $q$-periodic extension in $\bbD^\bbZ$.} on an open set in $\bbD^q$. Since the commutator in \eqref{eq:cmv:commMzalphabetaDef} is a real-analytic function of the variables $\{\Re \beta_j, \Im \beta_j: 1\le j \le q\}$, this implies that the commutator vanishes for all $\beta \in \bbD^q$. In particular, $M_z(\alpha)$ and $M_z(\beta)$ are simultaneously diagonalizable for all $\beta \in \bbD^q$. By choosing $\beta = (a,a,\ldots,a)$ for some $a \in \bbD$, we see that this implies $M_z(\alpha) = cI$ for a constant $c$, which  in turn forces $M_z(\alpha) = \pm I$ (because $M_z(\alpha) \in \SU(1,1)$), as desired.
\end{proof}

\begin{lemma} \label{lem:OPUC:resolventCover}
Suppose $\mathcal{E}_\alpha$ is a $q$-periodic extended CMV matrix. For all $\varepsilon>0$ and all $z \in \partial\bbD$, there exists $N \in \bbN$ and an $Nq$-periodic $\beta \in \bbD^\bbZ$ such that $\delta(\alpha,\beta) < \varepsilon$ and $z \notin \sigma(\mathcal{E}_\beta)$.
\end{lemma}

\begin{proof}
Let $z \in \partial \bbD$ and $\varepsilon$ be given. 
\medskip

\textbf{Case 1: \boldmath $\tr \, M_z(\alpha)\in \bbR \setminus[-2,2]$.} In this case, $z \notin \sigma(\calE_\alpha)$, so set $\beta = \alpha$.

\medskip
\textbf{Case 2: \boldmath $\tr \, M_z(\alpha) \in (-2,2)$}. In this case $M_z(\alpha) \neq \pm I$ so, by Proposition~\ref{prop:OPUC:noncommuting}, there exists a $q$-periodic $\gamma$ such that
\begin{equation}
[M_z(\alpha),M_z(\gamma)] \neq 0.
\end{equation}
and $\delta(\alpha, \gamma) < \varepsilon$. Taking $\varepsilon$ small enough, we may ensure that $M_z(\gamma)$ also has trace in $(-2,2)$.
By Lemma~\ref{lem:SU11:nonabSemiGroup} the semigroup generated by $M_z(\alpha)$ and $M_z(\gamma)$ contains a hyperbolic element. We may choose $L \in \bbN$, $k_1,k_2,\ldots,k_L \in \bbN$, and $\sigma_1,\sigma_2,\ldots ,\sigma_L\in\{\alpha,\gamma\}$ so that the matrix
\begin{equation}
\widetilde{M} = M_z(\sigma_L)^{k_L} M_z(\sigma_{L-1})^{k_{L-1}} \cdots M_z(\sigma_1)^{k_1}
\end{equation}
is hyperbolic. The desired perturbation of $\alpha$ is the corresponding concatenation of $\alpha$ and $\gamma$. That is, define $\beta$ by
\begin{equation}
\beta_n 
= (\sigma_\ell)_n \quad \sum_{j=1}^{\ell-1}k_j \leq n < \sum_{j=1}^\ell k_j,
\end{equation}
and extend $\beta$ to a $Kq$-periodic function, where $K = k_1+\cdots+k_L$.

\medskip

\textbf{Case 3: \boldmath $\tr \, M_z(\alpha) = \pm 2$.}  The entries of $M_z(\alpha)$ and its trace are real-analytic functions of $\{\Re\alpha_j, \Im \alpha_j : 1 \le j \le q\}$. Since the equality $\tr M_z(\alpha)= \pm 2$ does not hold identically in $\alpha$, the set where it holds has empty interior, again by the identity principle for real analytic functions. Thus, by an arbitrarily small perturbation of $\varphi_0$ we reduce to Case~1 or Case~2.
\end{proof}

We now have all the pieces that are needed.

\begin{proof}[Proof of Lemma~\ref{lem:OPUC:thinSpec}]
The crucial observation is supplied by Lemma~\ref{lem:OPUC:resolventCover}. With that lemma in hand, the proof follows from precisely the same arguments used to prove 
\cite[Lemma~5.3]{FillmanOng2017JFA}. 
\end{proof}

\bibliographystyle{abbrv}

\bibliography{EFGLbib}

\end{document}